\documentclass[12pt]{amsart}

\usepackage[utf8]{inputenc}
\usepackage{amssymb,amsmath,amsthm, url, comment}
\usepackage[all]{xy}
\begin{document}

\newtheorem{thm}{Theorem}[section]
\newtheorem{prop}[thm]{Proposition}
\newtheorem{lem}[thm]{Lemma}
\newtheorem{cor}[thm]{Corollary}

\theoremstyle{definition}
\newtheorem{defn}[thm]{Definition}
\newtheorem{remark}[thm]{Remark}
\newtheorem{remarks}[thm]{Remarks}
\newtheorem{example}[thm]{Example}
\newtheorem{question}[thm]{Question}
\newtheorem{reduction}[thm]{Reduction}
\newtheorem{conjecture}[thm]{Conjecture}
\newtheorem*{ack}{Acknowledgment}

\newcommand{\rd}{\operatorname{rd}}
\newcommand{\Char}{\operatorname{char}}
\newcommand{\Stab}{\operatorname{Stab}}
\newcommand{\ed}{\operatorname{ed}}
\newcommand{\trdeg}{\operatorname{trdeg}}
\newcommand{\Sym}{\operatorname{S}}
\newcommand{\ind}{\operatorname{ind}}
\newcommand{\Ass}{\operatorname{Ass}}
\newcommand{\Gr}{\operatorname{Gr}}
\newcommand{\Fields}{\operatorname{Fields}}
\newcommand{\Sets}{\operatorname{Sets}}
\newcommand{\Orth}{\operatorname{O}}
\newcommand{\GL}{\operatorname{GL}}
\newcommand{\SO}{\operatorname{SO}}
\newcommand{\PGL}{\operatorname{PGL}}
\newcommand{\SL}{\operatorname{SL}}
\newcommand{\Spec}{\operatorname{Spec}}
\newcommand{\Gal}{\operatorname{Gal}}
\newcommand{\rank}{\operatorname{rank}}
\newcommand{\Span}{\operatorname{Span}}
\newcommand{\Spin}{\operatorname{Spin}}
\newcommand{\Alt}{\operatorname{A}}
\newcommand{\Br}{\operatorname{Br}}
\newcommand{\Aut}{\operatorname{Aut}}
\newcommand{\Rad}{\operatorname{Rad}}
\newcommand{\Sp}{\operatorname{Sp}}
\newcommand{\lev}{\operatorname{lev}}
\newcommand{\Ker}{\operatorname{Ker}}

\newcommand{\bbC}{\mathbb{C}}
\newcommand{\bbG}{\mathbb{G}}
\newcommand{\bbA}{\mathbb{A}}
\newcommand{\bbP}{\mathbb{P}}
\newcommand{\bbZ}{\mathbb{Z}}
\newcommand{\bbQ}{\mathbb{Q}}
\newcommand{\bbF}{\mathbb{F}}

\author{Zinovy Reichstein}
\address{Department of Mathematics\\
	University of British Columbia\\
	Vancouver, BC V6T 1Z2\\Canada}
\email{reichst@math.ubc.ca}
\thanks{Zinovy Reichstein was partially supported by
	National Sciences and Engineering Research Council of
	Canada Discovery grant 253424-2017.}



\title[Hilbert's 13th Problem for Algebraic Groups]{Hilbert's 13th Problem for Algebraic Groups}

	\begin{abstract} The algebraic form of Hilbert's 13th Problem asks for the resolvent degree $\rd(n)$ of the general 
	polynomial $f(x) = x^n + a_1 x^{n-1} + \ldots + a_n$ of degree $n$, where $a_1, \ldots, a_n$ are independent variables.
	The resolvent degree is the minimal integer $d$ such that every root of $f(x)$ can be obtained in a finite number of steps, 
	starting with $\bbC(a_1, \ldots, a_n)$ and adjoining algebraic functions in $\leqslant d$ variables at each step.
	Recently Farb and Wolfson defined the resolvent degree $\rd_k(G)$ of any finite group $G$ and any base field $k$ 
	of characteristic $0$. In this setting $\rd(n) = \rd_{\bbC}(\Sym_n)$, where 
	$\Sym_n$ denotes the symmetric group. In this paper we define $\rd_k(G)$ for every algebraic group 
	$G$ over an arbitrary field $k$, investigate the dependency of this quantity on $k$ and
	show that $\rd_k(G) \leqslant 5$ for any field $k$ and any connected group $G$. The question of whether 
	$\rd_k(G)$ can be bigger than $1$ for any field $k$ and any algebraic group $G$ over $k$ 
	(not necessarily connected) remains open.
    \end{abstract}

\maketitle

\section{Introduction}

	The algebraic forms of Hilbert's 13th Problem asks for the resolvent degree $\rd(n)$, which is the smallest integer
	$d$ such that a root of the general polynomial
	\[ f(x) = x^n + a_1 x^{n-1} + \ldots + a_n \]
	can be expressed as a composition of algebraic functions of $d$ variables with complex coefficients. 
	It is known that $\rd(n) = 1$ when $n \leqslant 5$, and that $1 \leqslant \rd(n) \leqslant n - \alpha(n)$.
	where $\alpha(n)$ is an unbounded but very slow growing function of $n$. Classical upper bounds of this form have been recently
	sharpened by Wolfson~\cite{wolfson}, Sutherland~\cite{sutherland} and Heberle-Sutherland~\cite{hs}. On the other hand,
	it is not known whether or not $\rd(n) > 1$ for any
	$n \geqslant 6$. For a brief informal introduction to Hilbert's 13th Problem, see~\cite[Section 1]{reichstein-ems}.
	For a more detailed discussion, see~\cite{dixmier, farb-wolfson}. 
	
	Farb and Wolfson~\cite{farb-wolfson} defined the resolvent degree $\rd_k(G)$ for every finite group $G$ over an arbitrary base field $k$ of characteristic $0$. In this setting $\rd(n) = \rd_{\bbC}(\Sym_n)$, and it is not known whether or not $\rd_k(G)$ can ever be $> 1$.
	
	In this paper we extend their definition of $\rd_k(G)$ to an arbitrary algebraic group $G$ (not necesarily finite, affine or smooth) defined over an arbitrary field $k$. Our definition proceeds in three steps. First we define the level of a finite field extension (Definition~\ref{def.level}), then the resolvent degree of a functor (Definition~\ref{def.rd}), then the resolvent degree of a algebraic group (Definition~\ref{def.rd-group}). Out first main result is the following. 
	
\begin{thm} \label{thm.main3} Let $G$ be a connected algebraic group over a field $k$. 
Then 

\smallskip
(a) $\rd_k(G) \leqslant 5$.

\smallskip
(b) Moreover, if $G$ has no simple components of type $E_8$, then $\rd_k(G) \leqslant 1$.
\end{thm}

Note that Theorem~\ref{thm.main3} was announced (in a weaker form and without proof) in Section 8 of my survey~\cite{reichstein-ems}.
We will also investigate the dependence of $\rd_k(G)$ on the base field $k$. Our main results in this direction are Theorems~\ref{thm.main1} and~\ref{thm.main2} below.
	
	\begin{thm} \label{thm.main1} Let $G$ be an algebraic group defined over $k$. Then $\rd_k(G) = \rd_{k'}(G_{k'})$
    for any field extension $k'/k$.	
	\end{thm}
	
	The case where $k'$ is algebraic over $k$ is fairly straightforward (see Proposition~\ref{prop.functor2}(b)); 
	the main point here is that the extension $k'/k$ can be arbitrary. In particular, if $G$ is defined over $\mathbb Z$, when
	Theorem~\ref{thm.main1} tells us that $\rd_{k'}(G_{k'}) = \rd_k(G_k)$ for any two fields $k$ and $k'$ of the same characteristic.
	In arbitrary characteristic, we prove the following.
	
	\begin{thm} \label{thm.main2} Let $G$ be a smooth affine group scheme over $\mathbb Z$. Denote the connected component
	of $G$ by $G^0$. Assume that $G^0$ is split reductive and $G/G^0$ is finite over $\bbZ$. Let $k$ be a field of characteristic $0$.
	Then $\rd_k(G_k) \geqslant \rd_k(G_{k_0})$ for any other field $k_0$. 
	\end{thm}
	
 We will deduce Theorem~\ref{thm.main2} from a more general result, Proposition~\ref{prop.dvr}, which compares the resolvent degrees of
 the general and the special fibers of a group scheme over a discrete valuation ring. 
 Note that both Theorems~\ref{thm.main1} and~\ref{thm.main2} apply in the classical case, where $G$ is an abstract 
 finite group viewed as a group scheme over $\mathbb Z$, in particular to $G = \Sym_n$.

\smallskip
A key role in our proof of Theorem~\ref{thm.main3}(b) will be played by a theorem of Tits, which asserts that 
if $G$ is a simple group over $k$ of any type other 
than $E_8$ and $T \to \Spec(K)$ is a $G$-torsor, then $T$ can be ``split by radicals", i.e., $T$ splits over some 
radical extension of $K$; see Section~\ref{sect.proof-main3}. Tits asked whether or not the same is true for simple groups 
of type $E_8$. Using the arguments in Section~\ref{sect.proof-main3}, one readily sees that a positive answer to this question 
would imply the following.

\begin{conjecture} \label{conj.main5}
$\rd_k(G) \leqslant 1$ for any connected algebraic group $G$ over any field $k$.
\end{conjecture}

Theorem~\ref{thm.main3}(a) (or more precisely, Proposition~\ref{prop.simple}(a)),
may thus be viewed as a partial answer to Tits' question. Note that it is not known whether 
or not $\rd_k(G)$ can be $> 1$ for any field $k$ and any algebraic group $G$ defined 
over $k$ (not necessarily connected). 

The remainder of this paper is structured as follows. Sections~\ref{sect.prel1} and~\ref{sect.prel2}
are devoted to preliminary material on essential dimension
of finite-dimensional algebras and field extensions. In Section~\ref{sect.level} defines 
the level of a finite field extension and explores its elementary properties.
Section~\ref{sect.valuation} studies how the level changes under specialization.
Section~\ref{sect.level-d} introduces the level $d$ closure of a field. 
The resolvent degree of a functor is introduced in Section~\ref{sect.rd-functor}.
This notion parallels the notion 
of essential dimension of a functor, due to Merkurjev, Berhuy and Favi~\cite{berhuy-favi} but the type of functor we allow is more restrictive.
Much of the work towards proving 
Theorems~\ref{thm.main3} - \ref{thm.main2} is, in fact, done in the general setting of functors
in Sections~\ref{sect.base-change-alg} and~\ref{sect.base-change-arb}. Section~\ref{sect.rd-group}
introduces the notion of resolvent degree of an algebraic group. In Section~\ref{sect.rd-abelian-variety}
we study the resolvent degree of infinitesimal groups and abelian varieties. The proof of Theorem~\ref{thm.main1} 
is completed in Section~\ref{sect.proof-main1}, the proof of Theorem~\ref{thm.main2} in Section~\ref{sect.proof-main2}, and the proof of Theorem~\ref{thm.main3} in Sections~\ref{sect.upper-bounds} - \ref{sect.proof-main3}. 
In the last section we show that Conjecture~\ref{conj.main5} 
follows from a positive answer to a long-standing open question of Serre (Question~\ref{conj.serre}).

The main focus of this paper is on the aspects of the subject which have not been previously investigated: resolvent degree of
connected groups and dependence of resolvent degree on the base field. However, many of the preliminary results 
overlap with existing literature and some have classical roots. In particular, Section~\ref{sect.level} overlaps 
with~\cite[Section 2]{farb-wolfson}, Section~\ref{sect.rd-group} with \cite[Section 3]{farb-wolfson}.
Section~\ref{sect.level-d} elaborates on the short note of Arnold and Shimura~\cite[pp,~45-46]{browder}; there is also some overlap between
Section~\ref{sect.upper-bounds} and \cite[Section 4]{wolfson}. I have tried to indicate these connections throughout 
the paper. I have also included independent characteristic-free proofs for most background results, with the goal of making the exposition 
largely self-contained. The arguments in this paper are mostly algebraic and valuation-theoretic, with only a few exceptions 
(e.g., in Section~\ref{sect.upper-bounds}). I have not included references to classical literature; an interested 
reader can find them in~\cite{farb-wolfson}.


\section{Preliminaries on finite-dimensional algebras}
\label{sect.prel1}

Let $K$ be a field and $A$ be finite-dimensional $K$-algebra. We will say that $A$ descends to a subfield $K_0$ of $K$ if there exists 
a $K_0$-algebra $A_0$ such that $A \simeq_K A_0 \otimes_{K_0} K$. Here $\simeq_K$ stands for isomorphism of algebras over $K$.
We will sometimes say that $A/K$ descends to $A_0/K_0$.

\begin{lem} \label{lem.algebras1} Let $k \subset K$ be a field extension, $A$ a finite-dimensional $K$-algebra, and
$S$ a finite subset of $A$. Then $A/K$ descends to $A_0/K_0$ such that $K_0$ is finitely 
generated over $k$, $A_0$ is a $K_0$-subalgebra of $A$, and $S \subset A_0$. 
\end{lem}

\begin{proof} Choose a $K$-vector space basis $b_1, \ldots, b_n$ in $A$. Write
$\displaystyle b_i \cdot b_j = \sum_{h = 1}^n c_{ij}^h b_h$
for every $i$, $j = 1, \ldots, n$ and
$\displaystyle s = \sum_{h = 1}^n \alpha_s^h b_h$
for every $s \in S$. Let \[ K_0 = k(c_{ij}^h, \; \alpha_{s}^h \; | \; i, j, h = 1, \ldots, n \; \; \text{and} \; \;
s \in S ) \] and $L_0$ be the $K_0$-subsalgebra of $A$ generated by $b_1, \ldots, b_n$. Then one readily sees that
$K_0$ is finitely generated over $k$, the natural map
$A_0 \otimes_{K_0} K \to A$
is an isomorphism over $K$, and $S \subset A_0$.
\end{proof}

\begin{defn} \label{def.ed-algebra}
Let $K$ be a field containing $k$ and $A$ be a finite-dimensional $K$-algebra. The essential dimension $\ed_k(A/K)$ is
the minimal value of $\trdeg_k(K_0)$, where the minimum is taken over all intermediate fields $k \subset K_0 \subset K$ such that
$L/K$ descends to $K_0$.
\end{defn}

\begin{lem} \label{lem.algebras2} Let $K$ be a field containing $k$ and $A$ a finite-dimensional $K$-algebra. Then $\ed_k(A) < \infty$.
Moreover, $A/K$ descends to some $A_0/K_0$ such that $K_0$ is finitely generated over $k$ and $\ed_k(A) = \ed_k(A_0) = \trdeg_k(K_0)$. 
\end{lem}

\begin{proof} 
Descend $A/K$ to $A_1/K_1$ so that $d = \trdeg_k(K_1)$ is the smallest possible, i.e., $d = \ed_k(A)$.
Note that a priori $d$ is a non-negative integer or $\infty$.
By Lemma~\ref{lem.algebras1}, $A_1/K_1$ further descends to $A_0/K_0$, where $k \subset K_0 \subset K_1$ and $K_0$ is finitely generated over
$k$. By the minimality of $d$, $\ed_k(A) = \ed_k(A_1) = \ed_k(A_0) = \trdeg_k(K_0) = d$. Moreover, since $K_0$ is finitely generated over $k$,
$d < \infty$.
\end{proof}

\begin{lem} \label{lem.ed0} Let $k \subset k' \subset K$ be fields and $A$ be a finite-dimensional $K$-algebra.
Then

\smallskip
(a) $\ed_{k'}(A) \leqslant \ed_k(A)$.

\smallskip
(b) If $k'$ is algebraic over $k$, then $\ed_{k'}(A) = \ed_{k}(A)$.

\smallskip
(c) There exists an intermediate field $k \subset l_0 \subset k'$ such that $l_0$ is finitely generated over $k$ and
$\ed_{l}(A) = \ed_{k'}(A)$ for any $l_0 \subset l \subset k'$.
\end{lem}

\begin{proof}
(a) Suppose $A$ descends to a subfield $K_0 \subset K$ containing $k$ such that $\trdeg_k(K_0)$ as small
as possible, i.e., $\trdeg_k(K_0) = \ed_k(A)$. Then $A$ also descends to $k' K_0$, where the compositum is 
taken in $K$. Now $\ed_{k'}(A) \leqslant \trdeg_{k'}(k' K_0) \leqslant \trdeg_k(K_0) = \ed_k(A)$.

(b) In view of part (a), it suffices to show that $\ed_{k}(A) \leqslant \ed_{k'}(A)$.
Indeed, $A$ descends to some intermediate field $k' \subset K_0' \subset K$ such that 
$\trdeg_{k'} (K_0') = \ed_{k'}(A)$. If $k'$ is algebraic over $k$, then 
$\ed_{k}(A) \leqslant \trdeg_k(K_0') = \trdeg_{k'}(K_0) = \ed_{k'}(A)$.

(c) By Lemma~\ref{lem.algebras2}, $A/K$ descends to some $A_0/K_0$ such that
$k' \subset K_0 \subset K$, $\ed_{k'}(A) = \ed_{k'}(A_0) = \trdeg_{k'}(K_0)$ and $K_0$ is is generated by finitely many elements 
over $k'$, say $K_0 = k'(a_1, \ldots, a_m)$.

Let $x_1, \ldots, x_m$ be independent variables over $k'$. 
For each subset $I = \{ i_1, \ldots, i_r \}  \subset \{ 1, 2, \ldots, m \}$, such that the elements 
$a_{i_1}, \ldots, a_{i_r}$ are algebraically dependent over $k'$, choose a polynomial $0 \neq p_I(x_{i_1}, \ldots, x_{i_r}) 
\in k'[x_{i_1}, \ldots, x_{i_r}]$ such that $p_I(a_{i_1}, \ldots, a_{i_r}) = 0$.
Now choose an intermediate field $k \subset l_0 \subset k'$ such that $l_0$ is generated by the coefficients of 
the polynomials $p_I$ for every such $I$.
With this choice of $l_0$, any subset of $\{ a_1, \ldots, a_m \}$ which is algebraically dependent over $k'$
remains algebraically dependent in $l_0$. In other words, $\trdeg_{l_0}(K_0) = \trdeg_{k'}(K_0)$ and thus
\begin{equation} \label{e.l_0} \ed_{l_0} (A) \leqslant \trdeg_{l_0}(K_0) = \trdeg_{k'}(L_0) = \ed_{k'}(A). 
\end{equation}
By part (a), $\ed_{l_0}(A) \geqslant \ed_{l}(A) \geqslant \ed_{k'}(A)$ for any intermediate field $l_0 \subset l \subset k'$.
Now~\eqref{e.l_0} tells us that both of these inequalities are, in fact, equalities, as desired.
\end{proof}

\section{Preliminaries on field extensions}
\label{sect.prel2}

We will be particularly interested in the case where the finite-dimensional $K$-algebra $A$ is itself a field. 
In this case we will usually use the letter $L$ in place of $A$ and 
write $\ed_k(L/K)$ in place of $\ed_k(A)$.

\begin{lem} \label{lem.Galois} Let $k \subset K \subset L$ be field extensions such that $[L: K] < \infty$.

\smallskip
(a) If $K^{\rm sep}$ is the separable closure of $K$ in $L$, then $\ed_k(K^{\rm sep}/K) \leqslant \ed_k(L/K)$ and
$\ed_k(L/K^{\rm sep}) \leqslant \ed_k(L/K)$.

\smallskip
(b) If $L$ is separable over $K$, and $L^{\rm norm}$ is the normal closure of $L$, then $\ed_k(L/K) = \ed_k(L^{\rm norm}/K)$.

\smallskip
(c) Suppose $K \subset E \subset L$ is an intermediate extension. If $E$ is separable over $K$, then $\ed_k(E/K) \leqslant \ed_k(L/K)$.  
\end{lem}

\begin{proof} (a) Suppose $L/K$ descends to $L_0/K_0$. Denote the separable closure of $K_0$ in $L_0$ by $(K_0)^{\rm sep}$.
Then $K^{\rm sep}/K$ descends to $K_0^{\rm sep}/K_0$, $L/K^{\rm sep}$ descends to $L_0/(K_0)^{\rm sep}$ and part (a) follows.

(b) is proved in \cite[Lemma 2.3]{br}.

(c) In view of part (a), it suffices to show that $\ed_k(E/K) \leqslant \ed_k(K^{\rm sep}/K)$. 
In other words, we may replace $L$ by $K^{\rm sep}$ and thus assume without loss of generality that $L$ is separable over $K$. By (b), we may further replace  $L$ by its normal
closure over $K$ and thus assume that $L$ is Galois over $K$. Then $E = L^H$, where $H$ is a subgroup of $G = \operatorname{Gal}(L/K)$.
By~\cite[Lemma 2.2]{br}, $L/K$ descends to some $L_0/K_0$, where $k \subset K_0 \subset K$, $\trdeg_k(K_0) = \ed_k(L/K)$, and
$L_0$ is a $G$-invariant subfield of $L$. Then $E/K$ descends to $L_0^H/K_0$. This tells us that $\ed_k(E/K) \leqslant \trdeg_k(K_0) = \ed_k(L/K)$, as desired.
\end{proof} 

We will say that a field extension $L/K$ is simple if $[L:K] < \infty$ and $L$ is generated by one element over $K$.
In other words, $L \simeq K[x]/(f(x))$, where $f(x) \in K[x]$ is an irreducible polynomial over $K$.
By the Primitive Element Theorem, every finite separable extension is simple.

\begin{lem} \label{lem.descent1} Suppose a finite field extension $L/K$ descends to $L_0/K_0$. Then $L/K$ is simple if and only if $L_0/K_0$ is simple.
\end{lem}

\begin{proof} One direction is obvious: if $L_0 = K_0(a)$ is simple, then $L = K(a)$ is also simple.

To prove the converse, assume that $L/K$ is simple and set $n = [L:K] = [L_0: K_0]$.
If $K_0$ is a finite field, then so is $L_0$. In this case $L_0/K_0$ is separable and hence, simple. 
Thus we may assume that $K_0$ is infinite. It suffices to show that $L_0$ contains an element of degree $n$ over $K_0$. 
View $L_0$ as the set of $K_0$-points of the $n$-dimensional affine space $\bbA^n$, and
$L$ as the set of $K$-points of $\bbA^n$. Let $X \subset \bbA^n$ be the subscheme of $\bbA^n$ determined by the condition that
for $x \in \bbA^n_{K_0}$, $1, x, \ldots, x^{n-1}$ are linearly dependent. Here multiplication in $\bbA^n_{K_0}$ comes from identifying
$\bbA^n(K_0)$ with $L_0$. It is easy to see that
$X$ is a closed subscheme of $\bbA^n$ (given by the vanishing of a single $n \times n$ determinant) defined over $K_0$. 
Then $X(K_0)$ is the set of elements of $L_0$ of degree $\leqslant n-1$ over $K_0$ 
and $X(K)$ is the set of elements of $L$ of degree $\leqslant n -1$ over $K$.
We know that $L/K$ is simple; hence, $X \subsetneq \bbA^n$. That is, $U = \bbA^n \setminus X$ is a non-empty Zariski 
open subscheme of $\bbA^n$ defined over $K_0$. Since $K_0$ is an infinite field, we conclude that
$U(K_0) \neq \emptyset$. In other words, $L_0/K_0$ is simple, as claimed. 
\end{proof}

\begin{lem} \label{lem.ed1} Let $k \subset K \subset L$ be fields such that $L/K$ is simple.
Assume $K'/K$ is another field extension (not necessarily finite), 
and $L' = K' L$ be a compositum of $K'$ and $L$ over $K$. 
Then $\ed_k(L'/K') \leqslant \ed_k(L/K)$.
\end{lem}

Note that Lemma~\ref{lem.ed1} is immediate from the definition of $\ed_k(L/L)$
in the case, where $L' \simeq L \otimes_K K'$ or equivalently, $[L':K'] = [L:K]$. 
The only (slight) complication arises from the fact that $[L':K']$ may be smaller than $[L:K]$.

\begin{proof} Set $n = [L: K]$ and $d = \ed_k(L/K)$. Then $L/K$ descends to some intermediate field $k \subset K_0 \subset K$ 
such that $\trdeg_k(K_0) = d$. That is, there exists a field extension $L_0/K_0$ such that $L \simeq_K L_0 \otimes_{K_0} K$,
where $\simeq_{K_0}$ denotes an isomorphism of fields over $K_0$.

By Lemma~\ref{lem.descent1}, $L_0/K_0$ is simple. That is,
$L_0 \simeq_{K_0} K_0[x]/(f(x))$, where $f(x) \in K_0[x]$ is a polynomial of degree $n$, irreducible over $K_0$. 
Then $L \, \simeq_K \, K[x]/(f(x))$. 
Now let $f(x) = f_1(x) \ldots f_r(x)$ be an irreducible decomposition of $f(x)$ over $K'$. A compositum $L'$ of $L$ and $K'$ is isomorphic to $K'[x]/(f_i(x))$ for some $i$, say, $L' \,  \simeq_{K'} \, K'[x]/(f_1(x))$. Denote the degree of $f_1(x)$ by $n_1 = [L': K']$ and the roots of
$f_1$ in the algebraic closure of $K$ by $\alpha_1, \ldots, \alpha_{n_1}$. Since each $\alpha_i$ is a root of $f(x) \in K_0[x]$, each $\alpha_i$ 
algebraic over $K_0$. Hence, the coefficients of $f_1(x)$, being elementary 
symmetric polynomials in $\alpha_1, \ldots, \alpha_{n_1}$, are also algebraic over $K_0$. This shows that
$f_1(x) \in K_0^{\rm alg}[x]$, where $K_0^{alg}$ is the algebraic closure of $K_0$ in $K$. In other words, $L'/K'$ descends to 
$K_0^{\rm alg}$. Consequently, 
\[ \ed_k(L'/K')  \leqslant \trdeg_k(K_0^{\rm alg}) = \trdeg_k(K_0) = d = \ed_k(L/K), \]
as desired.
\end{proof}

\begin{lem} \label{lem.ed2} Let $k \subset K \subset L$ be fields such that $[L : K]< \infty$. Then there exist
intermediate extensions $K = K^{(0)} \subset K^{(1)} \subset \ldots \subset K^{(r)} = L$ such that
$K^{(i)}/K^{(i-1)}$ is simple and $\ed_k(K^{(i)}/K^{(i-1)}) \leqslant \ed_k(L/K)$
for every $i = 1, \ldots, r$.
\end{lem}

\begin{proof} Set $d = \ed_k(L/K)$. By definition, $L/K$ descends to $L_0/K_0$, where $k \subset K_0 \subset K$ and $\trdeg_k(K_0) = d$. 
Let $\alpha_1, \ldots, \alpha_r$ be
generators for $L_0$ over $K_0$ and set $K_0^{(i)} = K_0(\alpha_1, \ldots, \alpha_i)$ and $K^{(i)} = K(\alpha_1, \ldots, \alpha_i)$.
We obtain the following diagram, where
 \[ \xymatrix{ K = K^{(0)}  \ar@{-}[d] \; \ar@{^{(}->}[r] &  K^{(1)}  \ar@{-}[d] \; \ar@{^{(}->}[r] & 
 \ldots   \; \ar@{^{(}->}[r] &   K^{(r)} = L \ar@{-}[d]
 \\
K_0 = K_0^{(0)}  \; \ar@{^{(}->}[r] &  K_0^{(1)}  \; \ar@{^{(}->}[r] & 
\ldots   \; \ar@{^{(}->}[r] &  K_0^{(r)}  = L_0 }
 \]
By our construction, the extension $K^{(i)}/K^{(i-1)}$ is simple for each $i = 1, \ldots r$. Moreover, 
\[ \ed_k(K^{(i)}/K^{i-1}) \leqslant \trdeg_k(K_0^{(i-1)}) = d \]
because $K^{(i)}/K^{(i-1)}$ descends to $K_0^{(i)}/K_0^{(i-1)}$ for each $i$.
\end{proof}

\section{The level of a finite field extension}
\label{sect.level}

We now define the level of a field extension, following Dixmier~\cite[Section 2]{dixmier}. 

\begin{defn} \label{def.level} Let $k$ be a base field, $K$ be a field containing $k$, and $L/K$ be a field extension of finite degree.
I will say that $L/K$ is of level $\leqslant d$ if there exists a diagram 
of field extensions 
\begin{equation} \label{e.tower}
  \xymatrix{ & K_m \ar@{-}[d] & \\ 
    L \ar@{-}[ur] & \vdots \ar@{-}[d] &   \\
    & K_2 \ar@{-}[d] & \\
    & K_1 \ar@{-}[d] & \\
    & K_0 \ar@{-}[luuu] \ar@{=}[r] & K  }
\end{equation}
such that $[K_{i} : K_{i - 1}] < \infty$ and $\ed_k(K_{i}/K_{i - 1}) \leqslant d$ for every $i = 1, \ldots, m$. 
The level of $L/K$ is the smallest such $d$; I will denote it by $\lev_k(L/K)$. 
\end{defn}

\begin{remarks} \label{rem.def-level}
(1) The same notion was introduced by Brauer~\cite{brauer} (in characteristic $0$) under the name of resolvent degree. In this paper we will reserve the term ``resolvent degree" for the resolvent degree of an object of a functor; see Definition~\ref{def.rd}. 
If $L/K$ is a finite separable extension,  then we may view $L/K$ as an object of the functor of \'etale algebras, and the two notions coincide; see Example~\ref{ex.etale}.

(2) (cf.~\cite[Lemma 2.5.3]{farb-wolfson}) If $k \subset K \subset L' \subset L$ and $[L: K] < \infty$, then $\lev_k(L'/K) \leqslant \lev_k(L/K)$. 
Indeed, any tower~\eqref{e.tower} showing that $\lev_k(L/K) \leqslant d$
also shows that $\lev_k(L'/K) \leqslant d$.

\smallskip
(3) (cf.~\cite[Lemma 2.5.1]{farb-wolfson}) Taking $m = 1$ and $K_1 = L$ in~\eqref{e.tower}, we see that $\lev_k(L/K) \leqslant \ed_k(L/K)$. 

\smallskip
(iv) We may assume without loss of generality that each extension $K_i/K_{i-1}$ in the tower~\eqref{e.tower}
is simple. Indeed, by Lemma~\ref{lem.ed2}, we may replace $K_i/K_{i-1}$ by a sequence of simple extensions without 
increasing the essential dimension.

(5) Definition~\ref{def.level} formalizes the classical notion of composition of algebraic functions. 
If $K$ is a field of rational functions on some algebraic variety $X$ defined over $k$, then it 
is natural to think of $K_1$ as being generated by algebraic (multi-valued) functions on $X$ in $\leqslant \ed_k(K_1/K)$ variables, 
and $K_i$ as being generated by compositions of $i$ algebraic functions on $X$ in $\leqslant \lev_k(L/K)$ variables. 

(6) No examples where $\lev_k(L/K) > 1$ are known. 
\end{remarks} 

\begin{lem} \label{lem.base-change} 
Assume that $k \subset k' \subset K \subset L$ 
are fields and $[L: K] < \infty$. Then 

\smallskip
(a) (cf. \cite[Lemma 2.5.2]{farb-wolfson}) $\lev_k(L/K) \geqslant \lev_{k'}(L/K)$. 

\smallskip
(b) Moreover, equality holds if $k'$ is algebraic over $k$. 

\smallskip
(c) Furthermore, there always exists an intermediate field $k \subset l_0 \subset k'$
such that $l_0$ is finitely generated over $k$ and $\lev_{l}(L/K) = \lev_{k'}(L/K)$ for every field $l$ between $l_0$ and $k'$.
\end{lem}

\begin{proof} Choose a tower $K = K_0 \subset K_1 \subset \ldots \subset K_m$, as in Definition~\ref{def.level}
and apply parts (a), (b) and (c) of Lemma~\ref{lem.ed0}, respectively,
to each intermediate extension $K_i/K_{i-1}$. In part (c), let $l_i/k$ be a finitely generated field extension
obtained by applying Lemma~\ref{lem.ed0}(c) to 
$K_i/K_{i-1}$. Now set $l_0$ to be the compositum of $l_1, \ldots, l_m$ in $k'$ over $k$.
\end{proof}

\begin{lem} \label{lem.lev1} Assume that $k \subset K \subset L$ are fields and $[L:K] < \infty$ and let $K \subset K'$ 
be another field extension (not necessarily finite).
Then $\lev_k(K'L/ K') \leqslant \lev_k(L/K)$. Here $K' L$ denotes an arbitrary compositum of $K'$ and $L$ over $K$.
\end{lem} 

\begin{proof} Set $d = \lev_k(L/K)$ and choose
a tower $K = K_0 \subset K_1 \subset \ldots \subset K_m$ as in~Definition~\ref{def.level}. By Remark~\ref{rem.def-level}(4)
we may assume that each intermediate extension $K_i/K_{i-1}$ is simple. Now consider the tower
\[ K' = K_0' \subset K_1' \subset \ldots \subset K_m' , \]
Here $K_m' = K' K_m$ is some compositum of $K'$ and $K_m$, and for $i = 0, \ldots, m-1$,
$K_i' = K' K_i$ is the compositum of $K'$ and $K_i$ in $K_{m}'$. Since $K \subset L \subset K_m$,
$K'L$ embeds into $K' K_m$ over $K$. Since $K_i/K_{i-1}$ is simple,
Lemma~\ref{lem.ed1} tells us that $\ed_k(K_i'/ K_{i-1}') \leqslant d$. We conclude that $\lev_k(K'L/K') \leqslant d$.
\end{proof}

\begin{lem} \label{lem.algebraic}
Assume that $k \subset K \subset L$ are fields and $[L:K] < \infty$. 
Then $\lev_k(L/K) = 0$ if and only if $L$ embeds in a compositum  $\overline{k} K$ over $k$. 
In particular, if $k$ is algebraically closed, then $\lev_k(L/K) = 0$ if and only if $L = K$.
\end{lem}

\begin{proof} The second assertion is an immediate consequence of the first.

To prove the first assertion, suppose $L \subset \overline{k} K$. In other words, $L/K$ 
is generated by elements $\alpha_1, \ldots, \alpha_m \in L$ which are algebraic over $k$.
Consider the tower of simple extensions
\[ k = k_0 \subset k_1 \subset \ldots \subset k_m, \]
where $k_i = k(\alpha_1, \ldots, \alpha_s)$. Since $\trdeg_k(k_i) = 0$, 
we have $\ed_k(k_i/k_{i-1}) = 0$ for every $i = 1, \ldots, m$. 
Now consider the tower
\[ K = k_0 K \subset k_1 K \subset \ldots \subset k_m K = L . \]
By Lemma~\ref{lem.lev1}, $\ed_k(k_i K/k_{i-1} K) \leqslant \ed_k(k_i/k_{i-1}) = 0$.
Thus $\lev_k(L/K) = 0$.

Conversely, suppose $\lev_k(L/K) = 0$. Then there exists a tower~\eqref{e.tower} of field extensions such that
$L \subset K_m$ (over $K$) and $\ed_k(K_{i}/K_{i-1}) = 0$ for each $i$. Consequently, $K_{i}$ is generated over 
$K_{i-1}$ by elements that are algebraic over $k$. This implies that $K_m$ embeds in $\overline{k}K$ over $K$, and hence, so does $L$. 
\end{proof}

Recall that a finite field extension $L/K$ is called radical if there exists a tower \eqref{e.tower} 
such that $K_i = K_{i - 1}(\lambda)$, where $\lambda^{n_i} \in K$ for some $n_i \geqslant 1$, $i = 1, \ldots, m$. 

\begin{lem} \label{lem.radical} Let $K$ be a field containing $k$ and $L/K$ be a finite field extension.
Assume that $L/K$ is (a) solvable, (b) radical, (c) purely inseparable. Then $\lev_k(L/K) \leqslant 1$. 
\end{lem}

\begin{proof} 
By definition $L/K$ is solvable if there exists a tower $K = K_0 \subset K_1 \subset \ldots \subset K_m$, 
as in~\eqref{e.tower}, such that $L$ embeds into $K_m$ over $K$ and
$K_i$ is of the form $K_{i-1}(\lambda_i)$ for each $i = 1, \ldots, m$, where
$\lambda_i$ is a root of a polynomial of the form

\smallskip
(i) $x^{n_i} - a_i$ or
(ii) $x^{n_i} - x - a_i$ 
for some positive integer $n_i$ and $a_i \in K_{i-1}$,

\smallskip
\noindent
Note that (i) covers the case, where $\lambda_i$ is a root of unity ($a_i = 1$), 
and (ii) is only needed when $n_i = \operatorname{char}(k) > 0$.
In both cases $K_i = k(\lambda)K_{i-1}$ and thus 
\[ \ed_k(K_i/K_{i-1}) \leqslant \ed_k (k(\lambda)/k(a)) \leqslant 1. \]
Here the first inequality follows from Lemma~\ref{lem.ed1}. The second inequality is 
obvious, since $\trdeg_k ( k(a)) \leqslant 1$. Thus $\lev_k(L/K) \leqslant 1$. This proves (a).

(b) is proved by the same argument, except that case (ii) does not occur.

(c) follows from (b) because every purely inseparable extension is radical.
\end{proof}

\begin{lem} \label{lem.lev4} (cf.~\cite[Lemma 2.7]{farb-wolfson}) Let $K$ be a field containing $k$ and $L/K$ and $M/L$
be field extensions of finite degree. If $\lev_k(L/K) \leqslant d$ and $\lev_k(M/L) \leqslant d$, then
\[ \lev_k(M/K) \leqslant d . \]
\end{lem}

\begin{proof}  Choose a tower
$K = K_0 \subset K_1 \subset \ldots \subset K_m$ for $L/K$ as in~\eqref{e.tower}, and a similar tower
$L = L_0 \subset L_1 \subset \ldots \subset L_n$ for $M/L$. That is,
$L$ embeds into $K_m$ over $K$, $M$ embeds into $L_n$ over $L$, $\ed_k(K_i/K_{i-1}) \leqslant d$ 
and $\ed_k(L_j/L_{j-1}) \leqslant d$ for every $i = 1, \ldots, m$ and $j = 1, \ldots, n$. 
By Remark~\ref{rem.def-level}(4) we may assume that all intermediate extensions $K_i/K_{i-1}$ and $L_j/L_{j-1}$
are simple.
Let $\overline{K}$ be an algebraic closure of $K$.
Fix embeddings $K_m \hookrightarrow \overline{K}$ and $L_n \hookrightarrow \overline{K}$ and consider the tower of simple extensions
\[ K_0 \subset K_1 \subset \ldots \subset K_m = K_m L_0\subset K_m L_1 \subset \ldots \subset K_{m} L_n , \]
where $K_m L_i$ is the compositum of $K_m$ and $L_i$ in $\overline{K}$.
Clearly, $M \subset L_n \subset K_{m} L_n$. Thus it suffices to show that 

\smallskip
(i) $\ed_k(K_{i}/K_{i-1}) \leqslant d$ for every $i = 1, \ldots, m$ and 

\smallskip
(ii) $\ed_k(K_m L_j/ K_{m} L_{j-1}) \leqslant d$ for every $j = 1, \ldots, n$.

\smallskip
\noindent
(i) follows from our choice of the tower $K_0 \subset K_1  \subset \ldots \subset K_m$.
On the other hand,
by Lemma~\ref{lem.ed1}, $\ed_k(K_m L_j / K_m L_{j-1}) \leqslant \ed_k(L_j/ L_{j-1})$, and
by our choice of the tower $L_0 \subset L_1 \ldots \subset L_n$,
$\ed_k(L_j/ L_{j-1}) \leqslant d$ for every $j = 1, \ldots, n$. This proves (ii).
\end{proof}

\begin{lem} \label{lem.norm} (cf.~\cite[Lemma 2.11]{farb-wolfson})
Let $k \subset K \subset L$ be fields. Assume that the field extension $L/K$ is finite and separable.
Denote the normal closure of $L$ over $K$ by $L^{\rm norm}$. Then $\lev_k(L/K) = \lev_k(L^{\rm norm}/K)$.
\end{lem}

\begin{proof} By Remark~\ref{rem.def-level}(2), $\lev_k(L/K) \leqslant \lev_k(L^{\rm norm}/K)$. We will thus focus on proving the opposite inequality.

Set $d = \lev(L/K)$. By the Primitive Element Theorem, $L^{\rm norm} \simeq_K K[x]/f(x)$ for some irreducible polynomial $f(x) \in K[x]$.
Then $f(x)$ splits into a product of linear factors over $L^{\rm norm}$. Denote its roots in $L^{\rm norm}$ by $\alpha_1, \ldots, \alpha_n$.
Set $L_i = K(\alpha_1, \ldots, \alpha_i)$; in particular, $L_0 = K$. We claim that $\lev_k(L_i/L_{i-1}) \leqslant d$ for each $i = 1, \ldots, n$.
If we can prove this claim, then applying Lemma~\ref{lem.lev4} recursively, we obtain the desired inequality
\[ \lev_k(L^{\rm norm}/K) = \lev_k(L_n/K) \leqslant d = \lev_k(L/K). \]
It thus remains to prove the claim.  Since $L_i$ is a composite of $L_{i-1}$ and $K(\alpha_i) \simeq_K L$ for each $i$, Lemma~\ref{lem.lev1} tells us that
\[ \lev_k(L_i/L_{i-1}) = \lev(L_{i-1} K(\alpha_i) /L_{i-1}) \leqslant \lev_k(K(\alpha_i)/K) = \lev_k(L/K) = d , \]
as claimed.
\end{proof}

\begin{prop} \label{prop.lev-Galois}  (cf.~\cite[Lemma 2.12]{farb-wolfson})
Let $k \subset K \subset L$ be fields, where $[L:K]< \infty$. Assume that
$\lev_k(L/K) \leqslant d$. Then the tower \[ K = K_0 \subset \ldots \subset K_m \] of field extensions
in Definition~\ref{def.level} can be chosen to have the following additional properties.

\smallskip
(a) Each field extension $K_i/K_{i+1}$ is simple and either separable or purely inseparable.

\smallskip
(b) If $K_i/K_{i-1}$ is separable, then it is Galois.

\smallskip
(c) If $K_i/K_{i-1}$ is Galois, then $\Gal(K_{i}/K_{i-1})$ is a finite simple group.
\end{prop}

\begin{proof} We will start with a tower $K = K_0 \subset K_1 \subset \ldots \subset K_m$ of Definition~\ref{def.level}.
By Remark~\ref{rem.def-level}(4), we may assume that each intermediate extension $K_i/K_{i-1}$ is simple.
We will now modify this tower in three steps (a), (b) and (c), 
so that it acquires properties (a), (b), and (c) from the statement of the proposition, respectively. At each stage $m$ 
may increase and the fields $K_i$ may change, but every $K_i/K_{i-1}$ will remain simple, the largest field $K_m$ will either 
get larger or stay the same (and in particular, it will continue to contain $L$), and the maximal value of $\ed_k(K_i/K_{i-1})$ will 
not increase (so that it will remains $\leqslant d$). 

\smallskip
(a) Let $K^{\rm sep}_{i-1}$ be the separable closure of $K_{i-1}$ in $K_i$. If $K_i/K_{i-1}$ is neither separable nor
purely inseparable, i.e., $K_{i-1} \subsetneq K^{\rm sep}_{i-1} \subsetneq K_i$, we insert $K^{\rm sep}_{i-1}$ 
between $K_{i-1}$ and $K_i$.  Note that $K_i/K_{i-1}^{\rm  sep}$ is simple because $K_i/K_{i-1}$ is, and $K^{\rm sep}_{i-1}/K_{i-1}$
is simple by the Primitive Element Theorem. Now relabel $K_0, K_1, \ldots $ to absorb the newly inserted fields. By our construction
each $K_i/K_{i-1}$ is simple and either separable or purely inseparable.
The maximal value of $\ed_k(K_{i}/K_{i-1})$ does not increase by Lemma~\ref{lem.Galois}(a).

\smallskip
(b) If $K_1/K_0$ is purely inseparable, do nothing. If $K_1/K_0$ is separable, replace $K_1$ by its normal closure 
$K^{\rm norm}_1$ and $K_i$ by $K_1^{\rm norm} K_i$ for each $i \geqslant 2$. All newly created extensions $K_1^{\rm norm}/K_0$
and $K_1^{\rm norm}K_i/ K_1^{\rm norm} K_{i-1}$ ($i \geqslant 2$), remain simple and either separable or purely inseparable.
Moreover, for every $i \geqslant 2$, $\ed_k(K^{\rm norm}_1/K_0) = \ed_k(K_1/K_0)$ by Lemma~\ref{lem.Galois}(b) and 
$\ed_k(K_1^{\rm norm}K_i/ K_1^{\rm norm} K_{i-1}) \leqslant \ed_k(K_i/K_{i-1})$ by Lemma~\ref{lem.ed1}.

Now relabel $K_0, K_1, \ldots$ and do the same for the extension $K_2/K_1$.
That is, if $K_2/K_1$ is purely inseparable, then do nothing. If $K_2/K_1$ is separable, 
replace $K_2$ by its normal closure $K^{\rm norm}_2$, and $K_i$ by $K_2^{\rm norm} K_i$ for every $i \geqslant 3$.
Proceed recursively: do the same thing for the extension $K_3/K_2$, then (after suitably modifying $K_3, \ldots, K_m$) 
for the extension $K_4/K_3$, etc. When all of these modifications are completed, the resulting tower 
$K = K_0 \subset K_1 \subset \ldots \subset K_m$ will have properties (a) and (b).

\smallskip
(c) If $K_i/K_{i-1}$ is purely inseparable, do nothing. If it is Galois and $G = \Gal(K_i/K_{i-1})$ is simple, again do nothing.
If not - say if $G$ has a proper normal subgroup $N$ - insert $K_i^N$ between $K_{i-1}$ and $K_i$. By Lemma~\ref{lem.Galois}(c),
$\ed_k(K_i/K_i^N)$ and $\ed_k(K_i^N/K_{i-1})$ are both $\leqslant \ed_k(K_i/K_{i-1})$. Thus the maximal value of $\ed_k(K_i/K_{i-1})$
does not increase. Proceeding recursively, we arrive at a tower of field extensions satisfying (a), (b) and (c).
\end{proof}

\section{Extensions of valued fields}
\label{sect.valuation}

Throughout this section we will assume that

\smallskip
(1) $k \subset K \subset L$ are fields and $[L:K] < \infty$,

\smallskip
(2) $K$ and $L$ are complete relative to a discrete valuation $\nu \colon L^* \to \bbZ$. 

\smallskip
(3) We will denote the residue fields of $k$, $K$ and $L$ by $k_{\nu}$, $K_{\nu}$ and $L_{\nu}$, respectively. 

\smallskip
\noindent
Note that we do not require $k$ to be complete.
Our goal is to compare $\lev_k(L/K)$ to $\lev_{k_{\nu}}(L_{\nu}/K_{\nu})$. Our main result is as follows.

\begin{prop} \label{prop.lev-valuation}  In addition to notational convention (1) - (3), assume that the residue field $K_{\nu}$ is perfect.
Then $\lev_{k_{\nu}} (L_{\nu}/K_{\nu}) \leqslant \max \{ \lev_k(L/K), 1 \}$.
\end{prop}

Our proof of Proposition~\ref{prop.lev-valuation} will rely on the following lemma comparing the essential dimensions of
$L/K$ and $L_{\nu}/K_{\nu}$.

\begin{lem} \label{lem.ed-valuation} In addition to notational conventions (1), (2), (3), assume that
$L/K$ is a Galois extension, $H = \Gal(L/K)$ is a non-abelian finite simple group and
$L_{\nu}$ is separable over $K_{\nu}$. 
Then $\ed_{k_{\nu}} (L_{\nu}/K_{\nu}) \leqslant \ed_k(L/K)$.
\end{lem}

\begin{proof}[Proof of Lemma~\ref{lem.ed-valuation}]
First note that since $K$ is complete, $\nu$ is the unique valuation of $L$ lying over $\nu_{| \, K^*}$; 
see~\cite[Theorem II.3.1(ii)]{serre-lf}.
Thus $\nu$ remains invariant under the action of $H$, and this action descends to $L_{\nu}$.

Let $K_{\rm un}/K$ be the largest unramified subextension of $L/K$. Then $I = \Gal(L/K_{\rm un})$ is the inertia 
subgroup of $H = \Gal(L/K)$, i.e., the kernel of the $H$-action on $L_{\nu}$. 
In particular, $I$ is normal in $H$, and $L_{\nu}/K_{\nu}$ is an $(H/I)$-Galois extension; see~\cite[Corollary III.5.1]{serre-lf}. 
Since $H$ is simple, there are only two possibilities: either 

\smallskip
(i) $I = G$ and $K_{\rm un} = K$, i.e., $L/K$ is totally ramified, or

\smallskip
(ii) $I = 1$ and $K_{\rm un} = L$, i.e., $L/K$ is unramified.

\smallskip
\noindent
Let us consider these possibilities separately. Case (i) is straightforward. Here $L_{\nu} = K_{\nu}$, so that 
$\ed_{k_{\nu}} (L_{\nu}/K_{\nu}) = 0$ and hence, $\ed_{k_{\nu}} (L_{\nu}/K_{\nu}) \leqslant \ed_k(L/K)$, as desired.

In case (ii), $L_{\nu}/K_{\nu}$ is an $H$-Galois extension. 
Let $d = \ed_k(L/K)$. By definition, $L/K$ descends to $L_0/K_0$, where $k \subset K_0 \subset K$ and $\trdeg_k(K_0) = d$.
By \cite[Lemma 2.2]{br} we may assume that $L_0$ is invariant under $H$. Recall that $L = L_0 \otimes_{K_0} K$. Since $H$ acts faithfully on $L$
and trivially on $K$, it acts faithfully on $L_0$. The valuation $\nu$ restricts to an $H$-invariant discrete valuation 
on $L_0$, and the $H$-action on $L_0$ descends to an $H$-action on the residue field $(L_0)_{\nu}$.
Note however that a priori $L_0$ and $K_0$ may not be complete, and $L_0/K_0$ may be ramified. 

We claim that $H$ acts faithfully of $(L_0)_{\nu}$. Let us assume for a moment that this claim has been established.
Then $L_\nu/ K_\nu$  descends to $(L_0)_\nu/ (K_0)_\nu$, where $(K_0)_{\nu}$ denotes the residue field of $K_0$.
Indeed, the image of the natural map $(L_0)_{\nu} \otimes_{(K_0){\nu}} K_{\nu} \to L_{\nu}$ 
is surjective by the Galois correspondence, and hence, is an isomorphism, because $[(L_0)_{\nu}: (K_0)_{\nu}] = |H| = [L_{\nu}: K_{\nu}]$.
We thus conclude, that
\[ \ed_{k_{\nu}}(L_{\nu}/K_{\nu}) \leqslant \trdeg_{k_{\nu}} (K_{0})_{\nu} \leqslant \trdeg_k(K_0) = d, \]
as desired. Here the second inequality follows from \cite[lemma 2.1]{brosnan2018essential}, which is a special case of Abhyankar's Lemma. 

It remains to prove the claim. For each $d \geqslant 0$, let $L_0^{\geqslant d} = \{ a \in L^* \, | \, \nu(a) \geqslant d \} \cup \{ 0 \}$.
In particular, $L_0^{\geqslant 0}$ is the valuation ring of $\nu$ in $L_0$, $L_0^{\geqslant 1}$ is the maximal ideal, and
$L_0^{\geqslant 0}/L_0^{\geqslant 1}$ is, by definition, the residue field $(L_0)_{\nu}$. Let $I_d$ be the kernel of the $H$-action on $L_0^{\geqslant 0}/L_0^{\geqslant d + 1}$.
Then $I_0 \supset I_1 \supset I_2 \supset \ldots $ is a decreasing sequence of normal subgroups of $H$.
Since $H$ is simple, each $I_d$ is either all of $H$ or $1$. Our goal is to show that $I_0 = 1$. Assume the contrary: $I_0 = H$.
Consider two cases.

\smallskip
Case 1: $\Char((L_0)_{\nu}) = 0$. In this case $I_0 = H$ is a cyclic group; see~\cite[Corollary IV.2.2]{serre-lf} or \cite[Lemma 2.2(a)]{brosnan2018essential}. This contradicts our assumption that $H$ is non-abelian.

\smallskip
Case 2: $\Char((L_0)_{\nu}) = p > 0$. In this case $I_0 = H$ is of the form $P \ltimes C$, where $P$ is a $p$-group and $C$ is a cyclic group of
order prime to $p$; see~\cite[Corollary IV.2.4]{serre-lf}. Once again, this contradicts our assumption that $H$ is simple and non-abelian.
This completes the proof of the claim and thus of Lemma~\ref{lem.ed-valuation}.
\end{proof}

\begin{proof}[Proof of Proposition~\ref{prop.lev-valuation}] Let $d = \lev_k(L/K)$. By Definition~\ref{def.level} there exists
a tower 
\[ K = K_0 \subset K_1 \subset \ldots \subset K_m, \]
of finite field extensions, where
$L$ embeds in $K_m$ over $K$ and $\ed_k(K_i/K_{i-1}) \leqslant d$ for each $i = 1, \ldots, m$. Since $K$ is complete,
so are $K_1, \ldots, K_m$; see~\cite[Section II.2, Proposition 3]{serre-lf}. By Proposition~\ref{prop.lev-Galois}
we may assume that each $K_{i+1}/K_i$
is simple and either purely inseparable or Galois with $\Gal(K_{i+1}/K_i)$ a finite simple group. 
Passing to residue fields, we obtain a tower 
\[ K_{\nu} = (K_0)_{\nu} \subset (K_1)_{\nu} \subset \ldots \subset (K_m)_{\nu} \]
such that $L_{\nu}$ embeds into  $(K_m)_{\nu}$ over $k_{\nu}$. In view of Lemma~\ref{lem.lev4} it now suffices to show that
\begin{equation} \label{e.valuation}
\text{$\lev_{k_{\nu}}((K_{i})_{\nu}/ (K_{i-1})_{\nu}) \leqslant \max \{ d, 1 \}$ for each $i = 1, \ldots, m$.}
\end{equation}

If $K_{i}/K_{i-1}$ is purely inseparable, then $(K_i)_{\nu} / (K_{i-1})_{\nu}$ is again purely inseparable. By Lemma~\ref{lem.radical}(c),
$\lev_{k_{\nu}}((K_i)_{\nu} / (K_{i-1})_{\nu}) \leqslant 1$, and~\eqref{e.valuation} holds.

If $K_i/K_{i-1}$ is Galois, and $\Gal(K_i/K_{i-1})$ is simple and non-abelian, then 
\[ \lev_{k_{\nu}}((K_{i})_{\nu}/ (K_{i-1})_{\nu}) \leqslant 
\ed_{k_{\nu}}((K_{i})_{\nu}/ (K_{i-1})_{\nu}) \leqslant \ed_k(K_i/K_{i-1}) \leqslant d , \]
and~\eqref{e.valuation} follows. Here the first inequality is given by Remark~\ref{rem.def-level}(3) and the second by Lemma~\ref{lem.ed-valuation}.

It remains to consider the case, where $H_i = \Gal(K_i/K_{i-1})$ is abelian. 
Since we are assuming that $K_{\nu}$ is perfect, $(K_i)_{\nu}/ (K_{i-1})_{\nu}$ is a Galois extension,
where $\Gal \big( (K_i)_{\nu}/ (K_{i-1})_{\nu} \big)$ is a quotient of $\Gal(K_i/K_{i-1})$; see~\cite[Section 1.7, Proposition 20]{serre-lf}.
(Note that  the decomposition group is all of $\Gal(K_i/K_{i-1})$ here, because $K_{i}$ is complete.)
In particular, $(K_i)_{\nu}/ (K_{i-1})_{\nu}$ is an abelian (and hence, solvable) extension. 
Consequently, $\lev_{k_{\nu}}((K_{i})_{\nu}/ (K_{i-1})_{\nu}) \leqslant 1$ by Lemma~\ref{lem.radical}. We conclude that 
the inequality~\eqref{e.valuation} holds in this case as well.
\end{proof}

\section{The level $d$ closure of a field}
\label{sect.level-d}

\begin{defn} \label{def.d-closure}
Let $K$ be a field containing $k$
and $\overline{K}$ be an algebraic closure of $K$ and $d \geqslant 1$ be an integer. We define the level $d$ closure $K^{(d)}$ 
of $K$ in $\overline{K}$ to be the compositum of all intermediate extensions $K \subset L \subset \overline{K}$ such that
$[L: K] < \infty$ and $\lev_k(L/K) \leqslant d$. 
Clearly $K^{(1)} \subset K^{(2)} \subset K^{(3)} \subset \ldots$ . Up to isomorphism
(over $K$) the level $d$ closure $K^{(d)}$ depends only on $K$ and not on the choice of $\overline{K}$. We will say that
$K$ is closed at level $d$ if $K = K^{(d)}$, i.e., if $K$ has no non-trivial extensions of level $\leqslant d$.
\end{defn}

\begin{remark} \label{rem.0-closed}
If $d = 0$, then $K^{(0)} = \overline{k} K$, where $\overline{k}$ denotes the algebraic closure of $k$ and
the compositum is taken in $\overline{K}$. In particular, $K$ is closed at level $0$ if and only if $K$ contains an algebraic closure of $k$.
This follows directly from Lemma~\ref{lem.algebraic}.
\end{remark}

In the case, where $k$ is an algebraically closed field and $K = k(x_1, \ldots, x_n)$ is a purely transcendental extension,
Definition~\ref{def.d-closure} appeared in the short note of Arnold and Shimura in~\cite{browder}, pp.~45-46. 
In this section we will prove the following properties of level $d$ closure.  It seems likely that Arnold and Shimura 
had something like this in mind, through I have not encountered any explicit statements along these lines in the literature.

\begin{prop} \label{prop.level-d.1} Let $k \subset K \subset E$ be fields and $d \geqslant 0$ be an integer.

\smallskip
(a) Consider an intermediate field $K \subset L \subset \overline{K}$ such that  $[L:K] < \infty$.  
Then $\lev_k(L/K) \leqslant d$ if and only if $L \subset K^{(d)}$.

\smallskip
(b) $K^{(d)} \subset E^{(d)}$. Moreover, if $E$ is a finite extension of $K$ and $\lev_k(E/K) \leqslant d$, then equality holds,
$K^{(d)} = E^{(d)}$.

\smallskip
(c) $E^{(d)} = \bigcup \, E_{\rm f.g.}^{(d)}$, where the union is taken over
the intermediate fields $K \subset E_{\rm f.g.} \subset E$ with $E_{\rm f.g.}$ finitely generated over $K$.

\smallskip
(d) $(K^{(d)})^{(n)}= K^{(n)}$ for every $n \geqslant d$. In particular, $K^{(d)}$ is closed at level $d$.
\end{prop}

Our proof of Proposition~\ref{prop.level-d.1} will rely on the following lemma.

\begin{lem} \label{lem.lev2} Let $k \subset K \subset L$ be field extensions such that $[L:K] < \infty$. 
Then $L/K$ descends to some $L'/K'$, where $K'$ is finitely generated over $k$ and $\lev_k(L'/K') = \lev_k(L/K)$.
\end{lem}

\begin{proof} Set $d = \lev_k(L/K)$ and choose
a tower $K = K_0 \subset K_1 \subset \ldots \subset K_m$ of finite field extensions such that $L$ embeds into $K_m$ over $K$ 
and $\ed_k(K_i/K_{i-1}) \leqslant d$, as in Definition~\ref{def.level}. By Remark~\ref{rem.def-level}(4), we may assume that each
intermediate extension $K_i/K_{i-1}$ is simple.

By Lemma~\ref{lem.algebras2}, $K_i/K_{i-1}$ descends to some $E_i/F_{i-1}$, where
$E_{i} \subset K_i$, $k \subset F_{i-1} \subset K_{i-1}$, $F_{i-1}$ is finitely generated over $k$ and 
$\trdeg_k(F_{i-1}) = \ed_k(E_i/F_{i-1}) =  \ed_k(K_i/K_{i-1}) \leqslant d$. Let $G_{i-1}$ be a finite set of generators for $F_{i-1}$ over $k$.
By Lemma~\ref{lem.descent1}, $E_i/F_{i-1}$ is simple, say, $E_i = F_{i-1}(\alpha_i)$.

Similarly, by Lemma~\ref{lem.algebras1}, the field extension $L/K$ descends to $E'/F'$, where the intermediate field 
$k \subset F' \subset K$ is finitely generated over $k$. Let $H$ be a finite set of generators for $F'$ over $k$ and
$B$ be a $F'$-vector space basis for $E'$. These notations are summarized in the diagram below, where $\hookrightarrow$ indicates descent.
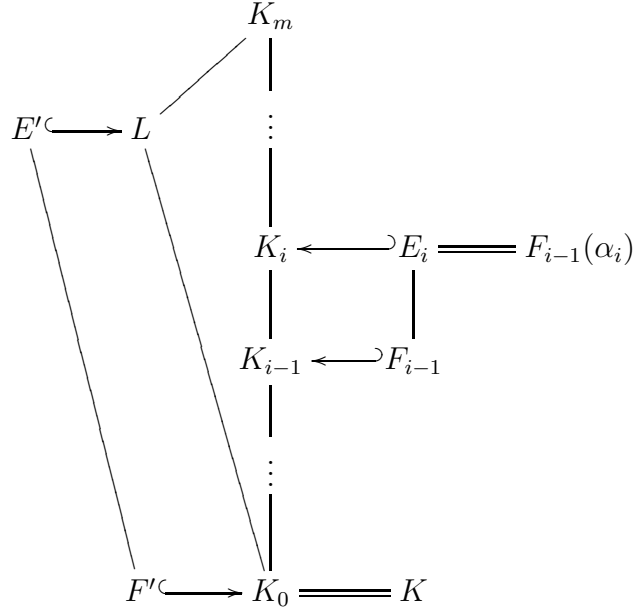
\[
  \xymatrix{ & & K_m \ar@{-}[d] &  & \\ 
 E' \ar@{^{(}->}[r] &    L \ar@{-}[ur] & \vdots \ar@{-}[d] &  &  \\
   &  & K_i \ar@{-}[d]  & \ar@{_{(}->}[l] E_i \ar@{-}[d] \ar@{=}[r] & F_{i-1}(\alpha_i) \\
   &    & K_{i-1} \ar@{-}[d]   & \ar@{_{(}->}[l] F_{i-1}  & \\
   &      & \vdots \ar@{-}[d] &  & \\
   & F' \ar@{^{(}->}[r] \ar@{-}[luuuu]  & K_0 \ar@{-}[luuuu] \ar@{=}[r] & K &  }
\]
By Lemma~\ref{lem.algebras1}, $K_m/K$ descends to some $K_m'/K'$ such 
that $k \subset K' \subset K$, $K'$ is finitely generated over $k$ and $K_m'$ contains the finite subset 
\[ G_0 \cup \ldots \cup G_{m-1} \cup \{ \alpha_1, \ldots, \alpha_m \} \cup H \cup B  \]
of $K_m$. Consider the tower
\begin{equation} \label{e.lev2a} K' = K_0' \subset K_1' \subset \ldots \subset K_m', \end{equation}
where $K_i' = K_m' \cap K_i$ for each $i$. Note that $K_{i-1}'$ contains $k$ and $G_{i-1}$ and hence, $k(G_{i-1}) = F_{i-1}$.
Moreover, since $K_i'$ contains $K_{i-1}'$ and $\alpha_i$, it also contains $F_{i-1}(\alpha_i) = E_i$.

Since $K_m/K$ descends to $K_m'/K'$, we have 
\begin{align} \label{e.lev2b}
[K_m': K_{m-1}'] \cdot [K_{m-1}': K_{m-2}'] \cdot \ldots \cdot [K_1' : K_0'] = 
[K_m': K'] = \\ \nonumber [K_m: K] = [K_m : K_{m-1}] \cdot [K_{m-1}: K_{m-2}] \cdot \ldots \cdot [K_1 : K_0] . 
\end{align}
On the other hand, since $\alpha_i \in K_i'$ has degree $[E_i: F_{i-1} ] = [K_i: K_{i-1}]$ over $K_{i-1}$,
it has degree $\geqslant [K_i: K_{i-1}]$ over $K_{i-1}'$. Thus $[K_i' : K_{i-1}'] \geqslant [K_i :K_{i-1}]$ for each $i$.
In view of~\eqref{e.lev2b}, we conclude that $[K_i': K_{i-1}'] = [K_i : K_{i-1}]$ for each $i$. In other words,
$K_i/K_{i-1}$ descends to $K_{i}'/K_{i-1}'$  which, in turn, descends to $E_i/F_{i-1}$. Thus
\[  \ed_k(K_i'/K_{i-1}') \leqslant \ed_k(E_i/F_{i-1}) \leqslant \trdeg_k(F_{i-1}) \leqslant d. \]
Finally, note that $K' = K_0'$ contains $H$ and thus $K'$ contains $k(H) = F'$. Set $L' = K_m' \cap L$.
Since $K_m'$ contains $B$, this tells us that $L/K$ descends to $L'/K'$. By our construction,
$L' = K_m' \cap L$ embeds into $K_m'$ over $K'$. 
The tower~\eqref{e.lev2a} now shows that $\lev_k(L'/K') \leqslant d$, as desired.
\end{proof}

\begin{proof}[Proof of Proposition~\ref{prop.level-d}] (a) If $\lev_k(L/K) \leqslant d$, then $L \subset K^{(d)}$ by the definition of $K^{(d)}$.
Conversely, if $L \subset K^{(d)}$ and $[L : K] < \infty$, then $L$ is contained in a compositum $L_1 L_2 \ldots L_n$
of finitely many finite extension $L_i/K$ such that $\lev_k(L_i/K) \leqslant d$ for each $i$. Using Lemmas~\ref{lem.lev1} and~\ref{lem.lev4}
recursively, we see that 
\[ \lev_k(L/K) \leqslant \lev(L_1 \ldots L_n/K) \leqslant d . \]

\smallskip
(b) Recall that $K^{(d)}$ is generated by finite extensions $L/K$ of level $\leqslant d$. 
In order to prove that $K^{(d)} \subset E^{(d)}$ it suffices to show that every such $L$ is contained in $E^{(d)}$. 
This follows from the inequality $\lev_k(LE/E) \leqslant \lev_k(L/K)$ of Lemma~\ref{lem.lev1}. 

Now suppose $E$ is a finite extension of $K$ and $\lev_k(E/K) \leqslant d$. We want to prove that in this case 
$E^{(d)} \subset K^{(d)}$. It suffices to show that every finite extension $M/E$ of level $\leqslant d$ 
lies in $K^{(d)}$, i.e., $\lev_k(M/K) \leqslant d$. This follows from Lemma~\ref{lem.lev4}.

\smallskip
(c) Set $U = \bigcup E_{\rm f.g.}^{(d)}$. By part (b), $E_{\rm f.g.}^{(d)} \subset E^{(d)}$ for each finitely generated field
$K \subset E_{\rm f.g.} \subset E$. Hence, $U \subset E^{(d)}$. To prove the opposite inclusion, we proceed in three steps.

\smallskip
(i) We reduce to the case, where $K = k$. Indeed, every intermediate field $k \subset E_0 \subset E$ such that $E_0$ is finitely generated over $k$
lies in $E_1 = K E_0$ which is finitely generated over $K$. By part (b), $E_0^{(d)} \subset E_1^{(d)} \subset U$. Thus
\[ \bigcup \, E_0^{(d)} \subset U \subset  E^{(d)} \]
where the first union is over finitely generated subextensions $k \subset E_0 \subset E$.
If we know that the first union is $E^{(d)}$, then both of these inclusions are equalities, and $U = E^{(d)}$, as desired.
From now on we will assume that $K = k$.

\smallskip
(ii) $U$ is a subfield of $E^{(d)}$. Indeed, suppose $x_1 \in E_1^{(d)}$ and $x_2 \in E_2^{(d)}$, where $k \subset E_i \subset E$ and $E_i$ is finitely 
generated over $k$ for $i = 1, 2$. Assume $x_1 \neq 0$. We want to show that $x_1 \pm x_2$, $x_1 \cdot x_2$ and $x_1^{-1}$ all lie in $U$. Indeed,
the composite $E_3 = E_1 E_2$ (in $E$) is also finitely generated over $k$. Hence $E_3^{(d)}$ is contained in $U$. By part (b),
$x_1 \in E_1^{(d)} \subset E_3^{(d)}$ and $x_2 \in E_2^{(d)} \subset E_3^{(d)}$. 
We conclude that $x_1 \pm x_2$, $x_1 \cdot x_2$ and $x_1^{-1} \in E_3^{(d)} \subset U$, as desired.

\smallskip
(iii) $U$ contains $E$. This is because $U$ contains $k(x)$ for every $x \in E$.

\smallskip
(iv) $U$ contains $E^{(d)}$.
It suffices to show that $U$ contains every finite extension $L/E$ such that $\lev_k(L/E) \leqslant d$.
Recall that by Lemma~\ref{lem.lev2}, $L/E$ descends to $L_0/E_0$ for some field $k \subset E_0 \subset E$ such that $E_0$ is finitely 
generated over $k$ and $\lev_k(L_0/E_0) = \lev_k(L/E) \leqslant d$. Thus $L_0 \subset E_0^{(d)} \subset U$. Since $U$ is a subfield of
$\overline{E}$ containing both $E$ and $L_0$, it contains $L = E L_0$.

\smallskip
(d) Part (b) tells us that $K^{(n)} \subset (K^{(d)})^{(n)}= (K^{(n)})^{(n)}$. Thus it suffices to show that $(K^{(n)})^{(n)} = K^{(n)}$, i.e.,
that $K^{(n)}$ is closed at level $n$. In other words, we may assume without loss of generality that $n = d$.

Let $E = K^{(d)}$. We want to show that $E^{(d)} = K^{(d)}$.
By part (c), it suffices to show that $L^{(d)} = K^{(d)}$ for every 
intermediate extension $K \subset L \subset E = K^{(d)}$, where $L$ is finitely generated (or equivalently, finite) over $K$.
Since $K \subset L \subset K^{(d)}$, part (a) tells us that $\lev_k(L/K) \leqslant d$. The desired equality, $L^{(d)} = K^{(d)}$,
is now given by the second assertion in part (b). 
\end{proof}

\begin{cor} \label{cor.level-1}
Suppose $K \in \Fields_k$ is closed at level $d \geqslant 1$. Then $K$ is perfect and solvably closed.
\end{cor}

\begin{proof}
Suppose $L/K$ is a solvable or purely inseparable extension. Our goal is to show that $K = L$. Indeed, by Lemma~\ref{lem.radical},
$\lev_k(L/K) \leqslant 1$. By Proposition~\ref{prop.level-d.1}(a),
$K \subset L \subset K^{(d)}$. Since $K$ is closed at level $d$, $K^{(d)} = K$ and thus
$K = L$.
\end{proof} 

\section{The resolvent degree of a functor}
\label{sect.rd-functor}

Following Merkurjev, Berhuy and Favi~\cite{berhuy-favi}, we will now define essential dimension for a broader class of objects, beyond finite 
field extensions. Let $k$ be a base field, and $\mathcal{F} \colon \Fields_k \to \Sets$ be a functor from the category 
of field extensions $K/k$ to the category of sets. All functors in this paper will be assumed to be covariant.
We think of $\mathcal{F}$ as specifying the type of object we are considering, and
$\mathcal{F}(K)$ as the set of objects of this type defined over $K$. Given a field extension $k \subset K \subset K'$, we think of 
the natural map $\mathcal{F}(K) \to \mathcal{F}(K')$ as base change. The image of $\alpha \in \mathcal{F}(K)$ under this map will be denoted by $\alpha_{K'}$.

\begin{defn} \label{def.ed-functor}
Any object $\alpha \in \mathcal{F}(K)$ in the image of the natural map $\mathcal{F}(K_0) \to \mathcal{F}(K)$ is 
said to {\em descend} to $K_0$. The essential dimension $\ed_k(\alpha)$ is defined as the minimal value of $\trdeg_k(K_0)$, 
where the minimum is taken over all intermediate fields $k \subset K_0 \subset K$ such that $\alpha$ descends to $K_0$.
\end{defn}

\begin{example} \label{ex.functor} Consider the functor $\operatorname{Alg} \colon \Fields_k \to \Sets'$, where
$\operatorname{Alg}(K)$ is the set of isomorphism classes of finite-dimensional $K$-algebras.
Here natural map $\operatorname{Alg}(K) \to \operatorname{Alg}(K')$ takes a $K$-algebra $A$ 
to the $K'$-algebra $K' \otimes_K A$.
Let $K \in \Fields_k$ and $A$ be a finite-dimensional $K$-algebra. 
If we view $A$ as an object in $\mathcal{F}(K)$, then
$\ed_k(A)$ given by Definition~\ref{def.ed-functor} is the same as $\ed_k(A)$ given by Definition~\ref{def.ed-algebra}.
\end{example}

Now let $\mathcal{F}$ be a functor from the category $\Fields_k$ of field extensions $K/k$ to the category $\Sets'$ 
of sets with a marked element. We will denote the marked element in $\mathcal{F}(K)$ by $1$ and will refer to it as being ``split". 
We will say that a field extension $L/K$ splits an object $\alpha \in \mathcal{F}(K)$ 
if $\alpha_L = 1$.  Let us assume that 
\begin{align}  \label{e.splittable} \text{for every field $K/k$ and every $\alpha \in \mathcal{F}(K)$, $\alpha$ can be} \\
\text{split by a field extension $L/K$ of finite degree.} \nonumber
\end{align} 
This is a strong condition on $\mathcal{F}$; in particular, it implies that
$\mathcal{F}(K) = \{ 1 \}$ whenever $K$ is algebraically closed.

\begin{defn} \label{def.rd}
Let $\mathcal{F} \colon \Fields_k \to \Sets'$ be a functor satisfying condition~\eqref{e.splittable}, $K/k$ be a field extension and 
$\alpha \in \mathcal{F}(K)$.

\smallskip
(a) The resolvent degree $\rd_k(\alpha)$ is the minimal integer $d \geqslant 0$ such that $\alpha$ is split by a field extension $L/K$ of level $d$
(or equivalently, of level $\leqslant d$).

\smallskip
(b) The resolvent degree $\rd_k(\mathcal{F})$ of the functor $\mathcal{F}$ is the maximal value of $\rd_k(\alpha)$, as $K$ ranges over all fields containing $k$ and $\alpha$ ranges over $\mathcal{F}(K)$.
\end{defn}

\begin{remarks}
(1) Note that the level $\lev_k(L/K)$ plays a similar role in Definition~\ref{def.rd} to the role played by the transcendence degree $\trdeg_k(K_0)$
in Definition~\ref{def.ed-functor}.

(2) Condition~\eqref{e.splittable} ensures that $\rd_k(\alpha)$ is finite for every $K \in \Fields_k$ and every $\alpha \in \mathcal{F}(K)$. On the other hand, 
$\rd_k(\mathcal{F})$ can a priori be infinite, even though no examples where $\rd_k(\mathcal{F}) > 1$ are known.
\end{remarks}

\begin{example} \label{ex.etale}
Consider the functor $\text{\'Et}_n \colon \Fields_k \to \Sets'$, where $\text{\'Et}(K)$ is the set of isomorphism 
classes of $n$-dimensional \'etale algebras $L/K$. Recall that
an $n$-dimensional \'etale algebra $L$ is a direct product of the form $L = L_1 \times \ldots \times L_r$, 
where each $L_i$ is a finite separable field extension of $K$ and $[L_1:K] + \ldots + [L_r: K] = n$.

\smallskip
(a) If $L/K$ is a separable field extension of degree $n$, and $[L]$ is its class in $\text{\'Et}_n(K)$, then $\rd_k([L]) = \lev_k(L/K)$.

\smallskip
(b) More generally, if $L = L_1 \times \ldots \times L_r$ is a direct product of separable extensions of $K$ as above, and $[L]$
is its class in $\text{\'Et}_n(K)$, then
$\rd_k([L]) = \max_{i = 1, \ldots, r} \lev_k(L_i/K)$.

\smallskip
(c) $\rd_k(\text{\'Et}) = \max \, \lev_k(L/K)$, where the maximum is taken over all separable field extensions $L/K$
of degree $\leqslant n$.
\end{example}

\begin{proof} (a) By the Primitive Element Theorem, $L \simeq_K K[x]/(f(x))$, where $f(x) \in K[x]$ is 
an irreducible separable polynomial of
degree $n$. A field extension $L'/K$ splits $[L]$ if and only if $f(x)$ splits as a product of linear factors over $L'$. Equivalently,
$L'$ splits $L$ if and only if $L'$ contains the normal closure $L^{\rm norm}$ of $L$ over $K$. By Remark~\ref{rem.def-level}(2),
\[ \rd_k([L]) = \min \{ \lev_k(L'/K) \; \big| \; L^{\rm norm} \subset L' \} =  \lev_k(L^{\rm norm}/K) . \]
On the other hand, by Lemma~\ref{lem.norm}, $\lev_k(L^{\rm norm}/K) = \lev_k(L/K)$.

(b) A field extension $L'/K$ splits $[L]$ if and only if it splits each $[L_i] \in \text{\'Et}_{[L_i:K]}(K)$. Hence,
by part (a), $\rd_k([L]) \geqslant \max_{i = 1, \ldots, r} \rd_k([L_i]) = \max_{i = 1, \ldots, r} \lev_k(L_i/K)$. To prove the opposite inequality,
take $L'$ to be the compositum of $L_i$ over $K$. Then $L'$ splits $[L]$. Moreover, combining Lemmas~\ref{lem.lev1} and~\ref{lem.lev4}, 
we obtain
\[ \rd_k([L]) \leqslant \lev_k(L'/K) \leqslant \max_{i = 1, \ldots, r} \lev_k(L_i/K). \]

(c) is an immediate consequence of (b).
\end{proof} 

\begin{lem} \label{lem.functor-3} Let $\mathcal{F} \colon \Fields_k \to \Sets'$ be a functor satisfying condition~\eqref{e.splittable}, $K/k$ be a field extension and $\alpha \in \mathcal{F}(K)$. Then 

(a) $\rd_k(\alpha_{K'}) \leqslant \rd_k(\alpha)$ for any field $K'$ containing $K$.

\smallskip
(b) $\rd_k(\alpha) \leqslant \ed_k(\alpha)$.

\smallskip
(c) $\rd_k(\mathcal{F}) \leqslant \ed_k(\mathcal{F})$.
\end{lem}

\begin{proof} (a) If $\alpha$ is split by a finite extension $L/K$ such 
that $\lev_k(L/K) = d$, then $\alpha_{K'}$ is split by the finite extension $K' L/K'$
of level $\lev_k(K' L/K') \leqslant d$; see Lemma~\ref{lem.lev1}. 

(b) Set $d = \ed_k(\alpha)$. Then $\alpha$ descends to $\alpha_0 \in \mathcal{F}(K_0)$ for some 
intermediate field $k \subset K_0 \subset K$ such that $\trdeg_k(K_0) =d$.
Since $\mathcal{F}$ satisfies condition~\eqref{e.splittable}, 
$\alpha_0$ is split by some finite extension $L_0/K_0$. Now
\[ \rd_k(\alpha) \leqslant \rd_k(\alpha_0) \leqslant \lev_k(L_0/K_0) \leqslant \ed_k(L_0/K_0) \leqslant d, \]
as desired. Here the first inequality follows from part (a), the second from the definition of
$\rd_k(\alpha_0)$, the third from Remark~\ref{rem.def-level}(3), and the fourth from the fact that $\trdeg_k(K_0) = d$.

(c) is an immediate consequence (b). 
\end{proof}

\begin{lem} \label{lem.functor-2} Let $k \subset k' \subset K$ be field extensions, Let $\mathcal{F} \colon \Fields_k \to \Sets'$ 
be a functor satisfying condition~\eqref{e.splittable} and $\alpha \in \mathcal{F}(K)$. Then

\smallskip
(a) $\rd_k(\alpha) \geqslant \rd_{k'}(\alpha)$. 

\smallskip
(b) Moreover, equality holds if $k'$ is algebraic over $k$. 

\smallskip
(c) Furthermore, there exists an intermediate field $k \subset l_0 \subset k'$
such that $l_0$ is finitely generated over $k$ and $\rd_{l}(\alpha) = \rd_{k'}(\alpha)$ for every field $l$ between $l_0$ and $k'$.
\end{lem}

\begin{proof} For every finite extension $L/K$ splitting $\alpha$, we have $\lev_k(L/K) \geqslant \lev_{k'}(L/K)$. Moreover, 
equality holds if $k'$ is algebraic over $k$; see Lemma~\ref{lem.base-change}. This proves (a) and (b).

For part (c), choose a splitting extension $L/K$ such that $d = \lev_{k'}(L/K)$ assumes its minimal possible value,  $d = \rd_{k'}(\alpha)$.
Now choose $l_0$ as in Lemma~\ref{lem.base-change}(c). Then for any intermediate field $l_0 \subset l \subset k'$, 
\[ \rd_l(\alpha) \leqslant \lev_l(L/K) = \lev_{k'}(L/K) = d = \rd_{k'}(\alpha). \]
Combining this inequality with the inequality of part (a), we conclude that $\rd_l(\alpha) = \rd_{k'}(\alpha)$.
\end{proof}

\begin{lem} \label{lem.functor-1} If $k$ is algebraically closed, then $\rd_k(\alpha) > 0$ for any $K \in \Fields_k$ and
any $1 \neq \alpha \in \mathcal{F}(K)$. In particular, $\rd_k(\mathcal{F}) = 0$ if and only if $\mathcal{F}$ 
is the trivial functor, i.e., if and only if $\mathcal{F}(K) = 1$ for every $K \in \Fields_k$. 
\end{lem}

\begin{proof} Immediate from Lemma~\ref{lem.algebraic}.
\end{proof}

\begin{lem} \label{lem.functor3} Let $\mathcal{F}_1, \mathcal{F}_2, \mathcal{F}_3$ be functors $\Fields_k \to \Sets'$ satisfying~\eqref{e.splittable}.

\smallskip
(a) Suppose $\mathcal{F}_1 \to \mathcal{F}_2 \to \mathcal{F}_3$ is an exact sequence~\footnote{This means that $\mathcal{F}_1(K) \to \mathcal{F}_2(K) \to \mathcal{F}_3(K)$ is an exact sequence 
in $\Sets'$ for every field $K/k$.}.
Then $\rd_k(\mathcal{F}_2) \leqslant \max \, \{ \rd_k(\mathcal{F}_1), \rd_k(\mathcal{F}_3) \}$.

\smallskip
(b) If a morphism $\mathcal{F}_1 \to \mathcal{F}_2$ of functors has trivial kernel, then $\rd_k(\mathcal{F}_1) \leqslant \rd_k(\mathcal{F}_2)$.

\smallskip
(c) If a morphism $\mathcal{F}_2 \to \mathcal{F}_3$ of functors is surjective, then $\rd_k(\mathcal{F}_3) \leqslant \rd_k(\mathcal{F}_2)$.

\smallskip
(d) If $1 \to \mathcal{F}_1 \to \mathcal{F}_2 \to \mathcal{F}_3 \to 1$ is a short exact sequence, then
$\rd(\mathcal{F}_2) = \max \, \{ \rd_k(\mathcal{F}_1), \rd_k(\mathcal{F}_3) \}$. In particular, 
$\rd_k (\mathcal{F}_1 \times \mathcal{F}_3) = \max \, \{ \rd_k(\mathcal{F}_1), \rd_k(\mathcal{F}_3) \}$.
\end{lem}

\begin{proof} (a) 
Suppose $\alpha \in \mathcal{F}_2(K)$ for some field $K/k$. Denote the image of $\alpha$ in $\mathcal{F}_3(K)$ by $\beta$. 
After passing to an extension $L/K$ of level $\leqslant \rd(\mathcal{F}_3)$, we may assume that $\beta$ is split.
Hence, $\alpha_L \in \mathcal{F}_2(L)$ is the image of some $\gamma \in \mathcal{F}_3(L)$. A further extension $L'/L$ of level $\leqslant \rd(\mathcal{F}_1)$ splits $\gamma$. 
Thus the composite extension $K \subset L \subset L'$
splits $\alpha$. We conclude that 
\[ \rd_k(\alpha) \leqslant \lev_k(L'/K) \leqslant \max \, \{ \lev_k(L/K) , \lev_k(L'/L) \} \leqslant \max \, \{  \rd_k(\mathcal{F}_1), \, \rd_k(\mathcal{F}_2)  \}, \]
where the inequality in the middle follows from Lemma~\ref{lem.lev4}.
Taking the maximum over all fields $K/k$ and all objects $\alpha \in \mathcal{F}_2(K)$, we conclude that
$\rd(\mathcal{F}_2) \leqslant \max \, \{ \rd_k(\mathcal{F}_1), \rd_k(\mathcal{F}_3) \}$.

\smallskip
(b) and (c): Apply part (a) to the exact sequences $1 \to \mathcal{F}_1 \to \mathcal{F}_2$ and $\mathcal{F}_2 \to \mathcal{F}_3 \to 1$,
respectively.

\smallskip
(d) For the first assertion combine the inequalities of (a), (b) and (c). The second assertion is a special case of the first with
$\mathcal{F}_3 = \mathcal{F}_1 \times \mathcal{F}_2$.
\end{proof}

\begin{prop} \label{prop.h2} Let $A$ be a diagonalizable group (i.e., a closed subgroup of the split torus $\bbG_m^d$)
defined over $k$. Then the functor $H^2(\ast, A)$ satisfies condition~\eqref{e.splittable} and
\[ \rd_k ( H^2(\ast, A) ) \leqslant 1. \] 
\end{prop}

\begin{proof} First let us consider the special case, where $A = \bbG_m$. Recall that $H^2(K, \bbG_m)$ is in a natural (functorial) bijection with the Brauer group $\Br(K)$. Thus it suffices to show that every central simple algebra $A$ over every $F \in \Fields_k$ can be split by 
a solvable extension of $K$. By the Primary decomposition Theorem we may assume without loss of generality that the index of $A$ is a prime power, $p^r$.
If $\Char(k) \neq p$, then the Merkurjev-Suslin Theorem tells us that $A$ can be split by a solvable extension of $K$; see \cite[Corollary 2.5.9]{gille-szamuely}. If $p = \Char(k)$, then by a theorem of Albert, $A$
is Brauer-equivalent to a cyclic algebra and thus can be split by a cyclic (and hence, once again, solvable) field extension of $K$.
This completes the proof in the case where $A = \bbG_m$. 

If $A = \mu_n$, then $H^2(\ast, \mu_n) \simeq {\, }_n \Br(K)$, where ${\, }_n \Br(K)$ is the $n$-torsion subgroup of $\Br(K)$, and the same argument applies.

In general we write $A$ as a direct product
$A_1 \times_k \ldots \times_k A_r$, where each $A_i$ is $k$-isomorphic to $\bbG_m$ or $\mu_{n}$ for some integer $n$. Then
$H^2(\ast, A) = H^2(\ast, A_1) \times \ldots \times H^2(\ast, A_r)$, and the desired conclusion follows from
Lemma~\ref{lem.functor3}(d).
\end{proof}

\begin{remark} \label{rem.h2}
If $A \neq 1$ in Proposition~\ref{prop.h2}, then equality holds: $\rd_k(H^2(\ast, A)) = 1$.

To prove this, we readily reduce to the case, where $A = \mu_n$ for some $n \geqslant 2$. In this case,
assume the contrary. Then for every $K \in \Fields_k$, every $\alpha \in H^2(K, \mu_n)$ can be split by a finite extension $L/K$
of level $0$. In particular, by Remark~\ref{rem.0-closed}, if $K$ contains $\overline{k}$, then $K$ is closed at level $0$, i.e., 
there are no non-trivial finite extensions $L/K$ of level $0$ and thus $H^2(K, \mu_n) = 1$. 
On the other hand, it is well known that if $K = \overline{k}(x, y)$, where $x$ and $y$ are variables, 
the symbol algebra $(x, y)_n$ represents a non-trivial class in $H^2(K, \mu_n)$, a contradiction.
\qed
\end{remark}

\begin{remark} \label{rem.hn}
Using the Norm Residue Isomorphism Theorem (formerly known as the Bloch-Kato Conjecture)  
in place of the Merkurjev-Suslin Theorem, one shows in the same manner that $\rd_k(H^d(*, A)) \leqslant 1$ for every $d \geqslant 1$
and that equality holds if $A \neq 1$.
\end{remark}

\section{Functors preserving direct limits}
\label{sect.base-change-alg}

In this section we will assume that our functor $\mathcal{F} \colon \Fields_k \to \Sets'$ respects direct limits.
Examples include Galois cohomology functors $H^1(\ast, G)$, where $G$ is an algebraic group over $k$, as well as $H^d(\ast, G)$ 
for every $d \geqslant 2$, if $G$ is abelian. For such functors $\mathcal{F}$ the study of resolvent degree can 
be facilitated by using the notion of level $d$ closure of of field introduced in Section~\ref{sect.level-d}. 

\begin{prop} \label{prop.level-d}
Assume that a functor $\mathcal{F} \colon \Fields_k \to \Sets'$ satisfies condition~\eqref{e.splittable} and respects 
direct limits. Let $K \in \Fields_k$ and $\alpha \in \mathcal{F}(K)$. Then

\smallskip
(a) $\rd_k(\alpha) \leqslant d$ if and only if $\alpha$ splits over $K^{(d)}$, i.e., $\alpha_{K^{(d)}} = 1$.

\smallskip
(b) $\rd_k(\mathcal{F}) \leqslant d$ if and only if $\mathcal{F}(K) = 1$ for every field $K$ closed at level $d$.

\smallskip
(c) Suppose $\rd_k(\alpha_{K^{(d)}}) \leqslant m$. Then $\rd_k(\alpha) \leqslant \max \{ d, \, m \}$. 

\smallskip
(d) Suppose $\rd_k(\beta) \leqslant m$ for every field $E \in \Fields_k$ closed at level $d$ and every $\beta \in \mathcal{F}(E)$. 
Then $\rd_k(\mathcal{F}) \leqslant \max \{ d, \, m \}$. 
\end{prop}

\begin{proof} (a) 
Suppose $\rd_k(\alpha) \leqslant d$. Then $\alpha$ splits over a finite extension $L$ of $K$ such that $\lev_k(L/K) \leqslant d$.
By definition of $K^{(d)}$, $L$ embeds into $K^{(d)}$ over $K$. Hence, $\alpha_{K^{(d)}} = 1$.
Conversely, suppose $\alpha$ splits over $K^{(d)}$. Since $\mathcal{F}$ respects direct limits, $\alpha$ splits over some subextension
$K \subset L \subset K^{(d)}$ such that $[L:K] < \infty$. By Proposition~\ref{prop.level-d.1}(a), 
$\lev_k(L/K) \leqslant d$. Thus $\rd_k(\alpha) \leqslant d$.

\smallskip
(b) Suppose $\rd_k(\mathcal{F}) \leqslant d$ and $K \in \Fields_k$ is closed at level $d$. By part (a), any $\alpha \in \mathcal{F}(K)$
splits over $K^{(d)}$. By Proposition~\ref{prop.level-d.1}(d), $K^{(d)} = K$ and thus $\alpha = 1$. This shows that $\mathcal{F}(K) = 1$.

Conversely, assume $\mathcal{F}(K) = 1$ whenever $K \in \Fields_k$ is closed at level $d$. Let $F$ be an arbitrary field containing
$k$ and $\alpha \in \mathcal{F}(F)$. By our assumption (with $K = F^{(d)}$), $\alpha_{F^{(d)}} = 1$. Since
$\mathcal{F}$ respects direct limits, $\alpha_E = 1$ for some $F \subset E \subset F^{(d)}$,
where $E$ is finitely generated over $F$, i.e., $[E : F] < \infty$. By Proposition~\ref{prop.level-d.1}(a), 
$\lev_k(E/F) \leqslant d$. Hence, $\rd_k(\alpha) \leqslant d$. 

(c) Let $n = \max \{ d, \, m \}$.  In view of part (a), our goal is to show that $\alpha_{K^{(n)}} = 1$. 
Set $E = K^{(d)}$. By Proposition~\ref{prop.level-d.1}(d), $E$ is closed at level $d$ and $E^{(n)} = K^{(n)}$. By our assumption,
$\rd_k(\alpha_{E}) \leqslant m$. By part (a), $\alpha_{E^{(m)}} = 1$. Since $E^{(m)} \subset E^{(n)}$,
we conclude that $\alpha_{K^{(n)}} = \alpha_{E^{(n)}} = 1$. 

(d) is an immediate consequence of (c).
\end{proof}

\begin{defn} \label{def.restriction} Let $\mathcal{F} \colon \Fields_k \to \Sets'$ be a functor. For any field $k'$ containing $k$,
we define $\mathcal{F}_{k'} \colon \Fields_{k'} \to \Sets'$ to be a restriction of $\mathcal{F}$ to $\Fields_{k'}$. In other words, $\mathcal{F}_{k'}(K)$ is only defined if $K$ contains $k'$, and for such $K$, $\mathcal{F}_{k'}(K) = \mathcal{F}(K)$.
\end{defn}

\begin{prop} \label{prop.functor2} Assume that a functor $\mathcal{F} \colon \Fields_k \to \Sets'$ satisfies condition~\eqref{e.splittable}.

\smallskip
(a) If $k'/k$ is a field extension, then the functor $\mathcal{F}_{k'}$ also satisfies condition~\eqref{e.splittable} 
and $\rd_k(\mathcal{F}) \geqslant \rd_{k'}(\mathcal{F}_{k'})$.

\smallskip
(b) Moreover, if $k'/k$ is an algebraic field extension and $\mathcal{F}$ respects direct limits, then $\rd_k(\mathcal{F}) = \rd_{k'}(\mathcal{F}_{k'})$.
\end{prop}

\begin{proof} Let $K \in \Fields_k$ and $\alpha \in \mathcal{F}(K)$. 

(a) The first assertion is obvious from Definition~\ref{def.restriction}. 
To prove the second assertion, it suffices to show that 
\begin{equation} \label{e.functor2}
\rd_k(\alpha) \geqslant \rd_k(\alpha_{k'K}) \geqslant \rd_{k'}(\alpha_{k'K}), 
\end{equation}
where $k'K$ is a compositum of $k'$ and $K$. Indeed, the maximal value of the left hand side over all $K \in \Fields_k$ and all $\alpha \in \mathcal{F}(K)$ is $\rd_k(\mathcal{F})$, where as the maximal value of the right hand side is $\rd_{k'}(\mathcal{F}_{k'})$.
The first inequality in~\eqref{e.functor2} follows from Lemma~\ref{lem.functor-3}(a) and the second 
from Lemma~\ref{lem.functor-2}(a). 

\smallskip
(b) Here it suffices to show that 
\begin{equation} \label{e.functor4b}
\rd_k(\alpha) = \rd_k(\alpha_{k'K}) = \rd_{k'}(\alpha_{k'K}). \end{equation}
The second equality follows from Lemma~\ref{lem.functor-2}(b). To prove the first inequality, it suffices to show that
\begin{equation} \label{e.functor4c}
K^{(d)} = (k'K)^{(d)}
\end{equation}
for every $d \geqslant 0$. Indeed, if we can prove this, then $\alpha$ is split by $K^{(d)}$ if and 
only if $\alpha_{k'K}$ is split by $(k'K)^{(d)}$, and the desired equality follows from Proposition~\ref{prop.level-d}(a).

To prove~\eqref{e.functor4c}, note that by Remark~\ref{rem.0-closed}, $K \subset k' K \subset K^{(0)}$. By Proposition~\ref{prop.level-d.1},
\[ K^{(d)} \subset (k' K)^{(d)} \subset (K^{(0)})^{(d)} = K^{(d)} , \]
and~\eqref{e.functor4c} follows.
\end{proof}

\begin{example} \label{ex.etale2} Let $\text{\'Et}_n \colon \Fields_k \to \Sets'$ be the functor of $n$-dimensional \'etale algebras 
introduced in Example~\ref{ex.etale}. If $d \geqslant \rd_k(\text{\'Et}_n)$, and $K \in \Fields_k$ is closed at level $d$,
then every polynomial of degree $\leqslant n$ splits into a product of linear factors over $K$.
\end{example}

\begin{proof} If $n = 1$, the assertion is vacuous, so we may assume that $n \geqslant 2$. One readily checks that 
the functor $(\text{\'Et}_n)_{\overline k}$ is non-trivial for any $n \geqslant 2$. Hence, 
\[ \rd_k(\text{\'Et}_n) \geqslant \rd_{\overline{k}} \big( (\text{\'Et}_n)_{\overline{k}} \big) \geqslant 1 , \]
where $\overline{k}$ denotes an algebraic closure of $k$, the first inequality follows from Proposition~\ref{prop.functor2}(a),
the second from Lemma~\ref{lem.functor-1}. Thus $d \geqslant 1$. By Corollary~\ref{cor.level-1}, $K$ is perfect. 

It remains to show that there does not exist an irreducible polynomial $f(x) \in K[x]$ of degree $m$ for any $2 \leqslant m \leqslant n$.
Indeed, assume the contrary. Then \[ E = K[x]/(f(x)) \times \underbrace{K \times \ldots \times K}_{\text{$n-m$ times}} \]
is a non-split \'etale algebra of degree $n$. On the other hand, $\text{\'Et}_n(K) = 1$ by Proposition~\ref{prop.level-d}(b), i.e., 
every \'etale algebra of degree $n$ over $K$ is split, a contradiction. 
\end{proof}

\begin{remark} \label{rem.truncated} 
Proposition~\ref{prop.functor2}(b) may fail if 

\smallskip
(a) the functor $\mathcal{F}$ is not required to respect direct limits or if

\smallskip
(b) the field $k'$ is not required to be algebraic over $k$.
\end{remark}

\begin{proof} Our counterexamples in parts (a) and (b) will both rely on the following construction. 
Let $\mathcal{F} \colon \Fields_k \to \Sets'$ be a functor and $\Lambda$ be a collection of fields $K \subset \Fields_k$
closed under inclusion. That is, if $L \in \Lambda$ and $K \subset L$, then $K \in \Lambda$. Set
$\mathcal{F}^{\Lambda} \colon \Fields_k \to \Sets'$ by
\[ \mathcal{F}^{\Lambda}(K) = \begin{cases} 
\text{$\mathcal{F}(K)$, if $K \in \Lambda$, and} \\          
\text{$\{ 1 \}$, if $K \not \in \Lambda$.}
\end{cases} \]
If $K \subset L$ be a field extension, the natural map $\mathcal{F}^{\Lambda}(K) \to \mathcal{F}^{\Lambda}(L)$ is defined to be 
the same as the natural map $\mathcal{F}(K) \to \mathcal{F}(L)$ if $L \in \Lambda$ and to be the trivial map (sending every element of $\mathcal{F}(K)$ to $1$)
if $L \not \in \Lambda$. It is easy to see that $\mathcal{F}^{\Lambda}$ is well defined. Moreover, if $\mathcal{F}$ satisfies condition~\eqref{e.splittable}, 
then so does $\mathcal{F}^{\Lambda}$. Informally, we think of $\mathcal{F}^{\Lambda}$ as a truncation of $\mathcal{F}$.

The starting point for both parts is a functor $\mathcal{F} \colon \Fields_k \to \Sets_k'$ which satisfies~\eqref{e.splittable},
respects direct limits, and such that 
$\rd_k(\mathcal{F}) \geqslant 1$. There are many examples of such functors, e.g., $\mathcal{F} =  H^2(\ast, \bbG_m)$; see Remark~\ref{rem.h2}.
Choose  a field $K \in \Fields_k$ and an object $\alpha \in \mathcal{F}(K)$ such that $\rd_k(\alpha) \geqslant 1$.
Since $\mathcal{F}$ respects direct limits, $\alpha$ descends to $\alpha_0 \in \mathcal{F}(K_0)$ for some
intermediate field $k \subset K_0 \subset K$ such that $K_0$ is finitely generated over $k$. By Lemma~\ref{lem.functor-3}(a), 
$\rd_k(\alpha_0) \geqslant \rd_k(\alpha) \geqslant 1$. After replacing $K$ by $K_0$ and $\alpha$ by $\alpha_0$,
we may assume that $K$ is finitely generated over $k$. 

(a) Consider the truncated functor $\mathcal{F}^{\Lambda}$, where 
\[  \Lambda = \{ K/k \; \big| \; \text{$K$ is finitely generated over $k$} \}. \]
Note that $\rd_k(\alpha) \geqslant 1$ whether we view $\alpha$ as an object 
in $\mathcal{F}$ or $\mathcal{F}^{\Lambda}$. On the other hand, if the algebraic closure $\overline{k}$ is not finitely generated 
over $k$ (e.g., if $k = \mathbb Q$), then no field containing $\overline{k}$ can be finitely generated over $k$. This tells us that 
the truncated functor $\mathcal{F}^{\Lambda}_{\overline{k}}$ is the trivial functor and consequently,
$\rd_{\overline{k}} \, (\mathcal{F}^{\Lambda}_{\overline{k}}) = 0$. We conclude that 
Proposition~\ref{prop.functor2}(b) fails for $\mathcal{F}^{\Lambda}$ if $k' = \overline{k}$.

(b) Set $m = \trdeg_k(K)$ and consider the truncated functor $\mathcal{F}^{\Lambda}$, where 
\[  \Lambda = \{ K/k \; \big| \;  \trdeg_k(K) \leqslant m \}. \]
The functor $\mathcal{F}^{\Lambda}$ continues to satisfy condition~\eqref{e.splittable} and to respect direct limits. 
If $\trdeg_k(k') > m$, then $\mathcal{F}^{\Lambda}_{k'}$ is trivial and 
thus $\rd_{k'}(\mathcal{F}^{\Lambda}) = 0$. On the other hand, $\rd_k(\alpha) \geqslant 1$ whether we view $\alpha$ as an object in $\mathcal{F}$ or 
$\mathcal{F}^{\Lambda}$. In summary, $\rd_k(\mathcal{F}^{\Lambda}) \geqslant 1$, $\rd_{k'}(\mathcal{F}^{\Lambda}_{k'}) = 0$, and  
Proposition~\ref{prop.functor2}(b) fails for $\mathcal{F}^{\Lambda}$.
\end{proof}

\section{Change of base field}
\label{sect.base-change-arb}

As we saw in Remark~\ref{rem.truncated}(b), Proposition~\ref{prop.functor2}(b) fails if $k'$ is not assumed to be algebraic over $k$. 
In this section we will show that under an additional condition on the functor $\mathcal{F}$, the equality of 
Proposition~\ref{prop.functor2}(b) can be (largely) salvaged for an arbitrary field 
extension $k'/k$. The condition we will impose on $\mathcal{F}$ is as follows:
\begin{align}  \label{e.condition*} 
\text{the natural map $\mathcal{F}(E) \to \mathcal{F}\big( E((t)) \big)$ has trivial} \\
\nonumber  
\text{kernel for every perfect field $E$ containing $k$.} 
\end{align}
Note that this is only slightly weaker than Condition (*) considered by Merkurjev in~\cite[Section 3]{merkurjev-gms}. The only difference between the two 
is that in~\cite{merkurjev-gms}, $E$ is not required to be perfect. As is pointed out in~\cite{merkurjev-gms}, this is a natural condition,
which is often satisfied.

\begin{prop} \label{prop.functor5}
Assume that a functor $\mathcal{F} \colon \Fields_k \to \Sets'$ satisfies conditions~\eqref{e.splittable} and~\eqref{e.condition*}
and respects direct limits. Then 
\[ \rd_{k'}(\mathcal{F}_{k'}) \leqslant \rd_{k}(\mathcal{F}) 
 \leqslant \max \{ \rd_{k'}(\mathcal{F}_{k'}), \,  1 \}  \]
for any field extension $k'/k$.
\end{prop}

The remainder of this section will be devoted to proving Proposition~\ref{prop.functor5}.
We begin with the following lemma.

\begin{lem} \label{lem.claim}
Assume that a functor $\mathcal{F} \colon \Fields_k \to \Sets'$ satisfies condition~\eqref{e.splittable} and
respects direct limits. Let $k'/k$ be a field extension, $K \in \Fields_k$ and
$\alpha \in \mathcal{F}(K)$. 
Then  there exists an intermediate field $k \subset l \subset k'$ such that
$l$ is finitely generated over $k$ and $\rd_l(\alpha_{lK}) = \rd_{k'}(\alpha_{k' K})$.
Here $k' K$ is some compositum of $k'$ and $K$ over $k$. The compositum $lK$ is taken in $k' K$.
\end{lem}

\begin{proof} For any intermediate field $k \subset l \subset k'$, we have
\[ \rd_l(\alpha_{lK}) \geqslant \rd_l(\alpha_{k'K}) \geqslant \rd_{k'}(\alpha_{k' K}) ; \]
see~\eqref{e.functor2}. Our goal is to show that the opposite inequality holds for a suitably 
chosen intermediate field $k \subset l \subset k'$, where $l$ is finitely generated over $k$.

Set $d = \rd_{k'}(\alpha_{k' K})$. By Lemma~\ref{lem.functor-2}(c) there exists an intermediate extension 
$k \subset l_0 \subset k'$ such that $l_0$ is finitely generated over $k$,
and $d = \rd_{l}(\alpha_{k' K})$ for any intermediate field $l_0 \subset l \subset k'$. 
After replacing $k$ by $l_0$ and $\alpha$ by $\alpha_{l_0 K}$, we may assume without loss of generality 
that $k = l_0$. In particular, $d = \rd_k(\alpha_{k' K})$. By Proposition~\ref{prop.level-d}(a), $\alpha$ splits over $(k'K)^{(d)}$.
Since $\mathcal{F}$ preserves direct limits, $\alpha$ splits over $L^{(d)}$ for some intermediate extension
$K \subset L \subset (k'K)$ such that $L$ is finitely generated over $K$. Any such $L$ is contained in $l K$ for
some intermediate field $k \subset l \subset k'$, where $l$ is finitely generated over $k$. Thus $\alpha$ splits over $(lK)^{(d)}$.
By Proposition~\ref{prop.level-d}(a) this implies that $\rd_l(\alpha_{lK}) \leqslant d$, as claimed.
\end{proof}

\begin{proof}[Proof of Proposition~\ref{prop.functor5}]
The first inequality $\rd_{k'}(\mathcal{F}_{k'}) \leqslant \rd_{k}(\mathcal{F})$ is proved in Proposition~\ref{prop.functor2}(a).
We will thus focus on proving the second inequality.
Let $K$ be a field containing $k$ and 
$\alpha \in \mathcal{F}(K)$.
Our goal is to show that
\begin{equation} \label{e.functor3b}
\rd_k(\alpha) \leqslant \max \{ \rd_{k'}(\alpha_{k' K}) , \, 1 \}.
\end{equation}
If we can prove this, then taking the maximum over all $K \in \Fields_k$ and all $\alpha \in \mathcal{F}(K)$,
we will obtain the desired inequality $\rd_k(\mathcal{F}) \leqslant \max \{ \rd_{k'}(\mathcal{F}_{k'}), \, 1 \}$.

We begin by reducing to the case where $k'$ is finitely generated over $k$.  
Indeed, choose $l$ as in Lemma~\ref{lem.claim}. That is, $l$ is finitely generated over $k$ and
$\rd_l(\alpha_{lK}) = \rd_{k'}(\alpha_{k' K})$. For the purpose of 
proving~\eqref{e.functor3b} we may now replace $k'$ by $l$.

From now on we will assume $k'$ is finitely generated over $k$.
Choose a transcendence basis $t_1, \ldots, t_n$ for $k'/k$ and set
$k_i = k(t_1, \ldots, t_i)$, so that $k'$ is algebraic over $k_n$. By~\eqref{e.functor4b},
\[ \rd_{k_n}(\alpha_{k_n K}) = \rd_{k_n}(\alpha_{k' K}) =  \rd_{k'}(\alpha_{k'K}). \]
Thus we may further replace $k'$ by $k_n$. It remains to show that
\begin{equation} \label{e.functor3c}
\rd_k(\alpha) \leqslant \max \{ \rd_{k(t)}(\alpha_{K(t)}) , \, 1 \},
\end{equation}
where $t$ is a variable. Indeed, applying this inequality recursively, we readily deduce~\eqref{e.functor3b}:
\[ \rd_{k}(\alpha) \leqslant \max \{ \rd_{k_1}(\alpha_{k_1 K}), \, 1 \} \leqslant \ldots \leqslant 
\max \{ \rd_{k_n}(\alpha_{k_n K}), \, 1 \}.  \] 
(Recall that here $k' = k_n$.)

The remainder of the proof will be devoted to establishing the inequality~\eqref{e.functor3c}. First observe that we
may assume without loss of generality that $K$ is closed at level $1$. 
Indeed, let $K^{(1)}$ be the level $1$ closure of $K$. By Proposition~\ref{prop.level-d}(c), 
\[ \rd_k(\alpha) \leqslant \max \{ \rd_k(\alpha_{K^{(1)}}), \, 1 \} \]
and by Lemma~\ref{lem.functor-3}(a),
\[ \rd_{k(t)}(\alpha_{K^{(1)}(t)}) \leqslant \rd_{k(t)}(\alpha_{K(t)}) , \] 
where $K^{(1)}$ is the level $1$ closure of $K$. These inequalities show 
that in the course of proving~\eqref{e.functor3c}, we may replace $K$ by $K^{(1)}$ and $\alpha$ by $\alpha_{K^{(1)}}$.
In other words, for the purpose of proving~\eqref{e.functor3c}, we may assume that $K$ is closed at level $1$. In particular, we may assume that
$K$ is a perfect field; see Corollary~\ref{cor.level-1}.

We now proceed with the proof of~\eqref{e.functor3c} under the assumption that $K$ is a perfect field.
First we observe that by Lemma~\ref{lem.functor-3}(a), $\rd_{k(t)}(\alpha_{K(t)}) \leqslant \rd_{k(t)} ( \alpha_{K((t))} )$. 
Thus we only need to show that
\begin{equation} \label{e.functor3d}
\rd_k(\alpha) \leqslant \max \{ \rd_{k(t)}(\alpha_{K((t))}) , \, 1 \},
\end{equation}
Set $d = \rd_{k(t)}(\alpha_{K((t))})$. By definition there exists 
a finite field extension $L/K((t))$ such that $\alpha_L = 1$ and $\lev_{k(t)} \big( L/K((t)) \big) = d$. 

The field $K((t))$ carries a natural discrete valuation $\nu \colon K((t))^* \to \bbZ$
with uniformizer $t$, trivial on $K$. Lift $\nu$ to a discrete 
valuation $L^* \to \dfrac{1}{e} \bbZ$, where $e$ is the ramification index.
By abuse of notation I will continue to denote this lifted valuation by $\nu$.
I will denote the residue field of $L$ relative to this valuation by $L_{\nu}$.

Note that since $K((t))$ is complete with respect to $\nu$, so is $L$.
Moreover, since $K$ is perfect, so is $L_{\nu}$.
Note also that we are in equal characteristic situation here: 
\[ \Char(L_{\nu}) = \Char(K) = \Char(k) = \Char(k(t)) = \Char(K((t))) = \Char(L). \]
By the Cohen Structure Theorem, $L$ is isomorphic to the field of Laurent series ${L}_{\nu}((s))$
in one variable over $L_{\nu}$.
Since $\alpha_L = 1$ and $L_{\nu}$ is perfect, the natural map 
\[ \mathcal{F}({L}_{\nu}) \to \mathcal{F}(L_{\nu}((s))) = \mathcal{F}({L}) \]
has trivial kernel, by our assumption~\eqref{e.condition*}. 
We conclude that $\alpha_{L_\nu} = 1$. In other words, ${L}_{\nu}/K$ is a splitting extension for $\alpha$.
By Proposition~\ref{prop.lev-valuation}, 
\[ \rd_k(\alpha) \leqslant \lev_k(L_{\nu}/K) \leqslant \max \{ \lev_k(L/K((t))), \, 1 \} = \max \{ d, \, 1 \}, \]
where $K$ is the residue field of $K((t))$.
Proposition~\ref{prop.lev-valuation} applies here because we are assuming that the residue field $K_{\nu} = K$ 
is perfect. This completes the proof of Proposition~\ref{prop.functor5}.
\end{proof}

\section{The resolvent degree of an algebraic group}
\label{sect.rd-group}

Let $G$ be an algebraic group over $k$, not necessarily affine, smooth or connected.
Of particular interest to us will be the functor $H^1(*, G)$ whose objects over $K$ are isomorphism classes
of $G$-torsors over $\Spec(K)$. Here torsors are assumed to be locally trivial in the flat (fppf) topology. If $G$ is 
smooth over $k$, this is equivalent to being trivial in the \'etale topology. 
For every field $K$ containing $k$, the set $H^1(K, G)$ has a marked element, represented by the split $G$-torsor $G_K \to \Spec(K)$, 
where $G_ K = G \times_{\Spec(k)} \Spec(K)$.
The functor $H^1(\ast, G)$ satisfies condition~\eqref{e.splittable}. 

\begin{defn} \label{def.rd-group} We define $\ed_k(G) = \ed_k (H^1(\ast, G) )$ and 
$\rd_k(G) = \ed_k (H^1(\ast, G))$.
\end{defn}

The essential dimension~$\ed_k(G)$ of an algebraic group $G/k$ has been much studied; see~\cite{merkurjev-survey, icm}. 
If $G$ is an abstract finite group (viewed as an algebraic group over $k$)
and $\Char(k) = 0$, our definition of $\rd_k(G)$ above coincides with the definition given by Farb and Wolfson~\cite{farb-wolfson}. 
To the best of my knowledge, $\rd_k(G)$ has not been previously investigated for other algebraic groups $G/k$.

In view of Proposition~\ref{prop.functor2}(b), passing from $G$ to $G_{\overline{k}}$ 
does not change the resolvent degree. Thus from now on we will assume that $k$ is algebraically closed.

\begin{example} \label{ex.etale3} The functor $\text{\'Et}_n$ introduced in Example~\ref{ex.etale} 
is isomorphic to $H^1(\ast, \Sym_n)$ and thus $\rd_k(\text{\'Et}_n) = \rd_k(\Sym_n)$. By Example~\ref{ex.etale},
\[ \rd_k(\Sym_n) = \max \, \lev_k(L/K) . \]
The algebraic form of Hilbert's 13th Problem asks for the value of $\rd(n) = \rd_{\bbC}(\Sym_n)$.
\end{example}

\begin{remark} \label{rem.motivation} Recall that the classical definition of $\rd(n)$ is motivated by
wanting to express a root of a general polynomial $f(x) = x^n + a_1 x^{n-1} + \ldots + a_n$ as a composition 
of algebraic functions in $\leqslant d$ variables. This is equivalent to finding the smallest integer $d$ such 
that the $0$-cycle in $\mathbb A^1_K$ given by $f(x) = 0$ has an $L$-point,
for some field extension $L/K$ of level $\leqslant d$. If $G$ is an algebraic group over $k$, $K$ is a field containing $k$
and $T \to \Spec(K)$ is a $G$-torsor, then our definition of $\rd_k(T)$ retains this flavor. 
Indeed, saying that $T$ is split by $L$ is equivalent to saying that $T$ has an $L$-point. 
\end{remark}

\begin{remark} \label{rem.group2.5} (cf.~\cite[Lemma 3.2]{farb-wolfson})
Let $G$ be an algebraic group defined over a field $k$, $K$ be a field containing $k$, and
$\alpha \colon T \to \Spec(K)$ be a $G$-torsor. Setting $\mathcal{F} = H^1(\ast, G)$ in Lemma~\ref{lem.functor-3}, we obtain 
the inequalities $\rd_k(\alpha) \leqslant \ed_k(\alpha)$ and $\rd_k(G) \leqslant \ed_k(G)$.
\end{remark} 

\begin{remark} \label{rem.subgroup} (cf.~\cite[Lemma 3.13]{farb-wolfson})
Let $G$ be a finite group and $H$ be a subgroup. We will view $G$ and $H$ as algebraic groups over $k$.
The long exact sequence in Galois cohomology associated to $1 \longrightarrow H \stackrel{i}{\longrightarrow} G$ 
(see~\cite[Section I.5.4]{serre1997galois}) readily show that the induced morphism
$i_* \colon H^1(\ast, H) \to H^1(\ast, G)$ has trivial kernel. By Lemma~\ref{lem.functor3}(b), we conclude that 
$\rd_k(H) \leqslant \rd_k(G)$.
\end{remark}

\begin{example} \label{ex.solvable-group} (cf.~\cite[Corollary~3.4]{farb-wolfson})
If $G$ is a solvable finite group, then $\rd_k(G) \leqslant 1$. Indeed, every element of $H^1(K, G)$ can be split by a solvable extension $L/K$, and a solvable extension has level $\leqslant 1$ by Lemma~\ref{lem.radical}(a).
Moreover, if we further assume that $G \neq 1$, then $\rd_k(G) = 1$. This follows from Lemma~\ref{lem.functor-1} and Proposition~\ref{prop.functor2}(b).
\end{example}

Recall that an algebraic group $G$ defined over a field $k$ is called special if $H^1(\ast, G)$ is the trivial functor, i.e.,
$H^1(K, G) = 1$ for every field $K$ containing $k$. This notion 
is due to Serre~\cite{serre-special}. (Note that \cite{serre-special} is reprinted in~\cite{serre-reprinted}.)

\begin{lem} \label{lem.group1} (a) Let $G$ be an algebraic group over an algebraically closed field $k$. Then

\smallskip
(a) $G$ is special if and only if $\rd_k(G) = 0$.

\smallskip
(b) If $G$ is connected and solvable, then $\rd_k(G) = 0$.

\smallskip
(c) If $G$ is the special linear group $\SL_n$ or the symplectic group $\Sp_{2n}$, then $\rd_k(G) = 0$ for any $n \geqslant 1$.
\end{lem}

\begin{proof} (a) By Lemma~\ref{lem.functor-1}, $\rd_k(G) = 0$ if and only if $H^1(\ast, G)$ is the trivial functor, i.e., if and only
if $G$ is special.

(b) Every connected solvable group is special; see~\cite[Section 4.4(a)]{serre-special}.

(c) $\SL_n$ and $\Sp_{2n}$ are special; see~\cite[Section 4.4(b) and (c)]{serre-special}. 
\end{proof}

We now record several simple but useful observations about the behavior of resolvent degree in exact sequences of groups.

\begin{prop} \label{prop.group2} Consider a short exact sequence of algebraic groups
\begin{equation} \label{e.exact} 
1 \to A \to B \to C \to 1 
\end{equation}
defined over a field $k$. Then

\smallskip
(a) $\rd_k(B) \leqslant \max \, \{  \rd_k(A), \, \rd_k(C)  \}$. 

\smallskip
(b) If $B$ is isomorphic to the direct product $A \times C$, then  $\rd_k(B) = \max \, \{  \rd_k(A), \, \rd_k(C)  \}$.

\smallskip
(c) If $G$ is a diagonalizable algebraic group over $k$, then $\rd_k(G) \leqslant 1$.

\smallskip
(d) Suppose \eqref{e.exact} is a central short exact sequence and $A$ is diagonalizable over $k$.
Then \[ \text{$\rd_k(B) \leqslant \max \, \{ \rd_k(C),  1 \}$ and $\rd_k(C) \leqslant \max \, \{ \rd_k(B),  1 \}$.} \]
\end{prop}

\begin{proof} (a) follows from Lemma~\ref{lem.functor3}(a) applied to the exact sequence of functors
$H^1(\ast, A) \to H^1(\ast, B) \to H^1(\ast, C) $
induced by~\eqref{e.exact}.

(b) follows from Lemma~\ref{lem.functor3}(d), since in this case
the functor $H^1(\ast, B)$ is isomorphic to $H^1(\ast, A) \times H^1(\ast, C)$.

(c) Write $G$ as a product $G_1 \times \ldots \times G_r$, where each $G_i$ is isomorphic 
either to $\bbG_m$ or $\mu_n$ for some $n \geqslant 2$. By part (b), it suffices to show that $\rd_k(G_i) \leqslant 1$
for each $i$. Now recall that $\rd_k(\bbG_m) = 0$ by Lemma~\ref{lem.group1}(b). On the other hand, 
$\rd_k(\mu_n) \leqslant \ed_k(\mu_n)$ by Remark~\ref{rem.group2.5}, and $\ed_k(\mu_n) = 1$ for every $n  \geqslant 2$; see, 
e.g.,~\cite[Example 3.5]{merkurjev-survey}.

(d) To prove the first inequality, combine parts (a) and (c).
The second inequality follows from Lemma~\ref{lem.functor3}(a) applied to the exact sequences of functors
$H^1(\ast, B) \to H^1(\ast, C) \to H^2(\ast, A) $
induced by~\eqref{e.exact}. Recall that $\rd_k ( H^2(*, A) ) \leqslant 1$ by Proposition~\ref{prop.h2}. 
\end{proof}

\begin{cor} \label{cor.group3}
Let $G$ be a connected reductive affine algebraic group, $T$ be a split maximal torus of $G$, 
$N$ be the normalizer of $T$ in $G$, and $W = N/T$ be the Weyl group. Then 
\[ \rd_k(W) \geqslant \rd_k(N) \geqslant \rd_k(G).  \]
\end{cor} 

\begin{proof} By~\cite[Corollary 5.3]{chernousov2008reduction}
the natural morphism $H^1(K, N) \to H^1(K, G)$ is surjective; 
see also~\cite[Lemma III.4.3.6]{serre1997galois}. 
Hence, $\rd_k(N) = \rd(H^1(\ast, N)) \geqslant \rd_k(H^1(\ast, G)) = \rd_k(G)$
by Lemma~\ref{lem.functor3}(c). 

The inequality $\rd_k(W) \geqslant \rd_k(N)$ follows from Proposition~\ref{prop.group2}(a) 
applied to the exact sequence $1 \to T \to N \to W \to 1$. Note that by Lemma~\ref{lem.group1}(b),
$\rd_k(T) = 0$.
\end{proof}

\section{The resolvent degree of an abelian variety}
\label{sect.rd-abelian-variety}

In this section we will assume that the base field $k$ is algebraically closed. 
This assumption is harmless in view of Proposition~\ref{prop.functor2}(b).

Recall that for every algebraic group $G$ defined over $k$, there exists a smooth (i.e., reduced) 
subgroup $G_{\rm red}$ such that $G(k) = G_{\rm red}(k)$; see~\cite[Exp.VI$_A$, Section 0.2]{SGA3I}. 

\begin{lem} \label{lem.non-reduced} Let $K$ be a field containing $k$ and let $i \colon G_{\rm red} \hookrightarrow G$ 
be the natural inclusion and  $i_* \colon H^1(K, G_{\rm red}) \to H^1(K, G)$ be the induced map in cohomology. Then

\smallskip
(a) $i_*$ is injective.

(b) If $K$ is a perfect field, then $i_*$ is bijective.
\end{lem}

\begin{proof} Let $\gamma \in H^1(K, G)$. 
By \cite[I.5.4, Corollary 2]{serre1997galois}, the fiber of $(i_*)^{-1}(\gamma)$ may be identified with the set of orbits of
${\, }_{\gamma}G(K)$ in $({\,}_{\gamma}G/{\, }_{\gamma} G_{\rm red})(K)$.~\footnote{In~\cite{serre1997galois} 
only \'etale cohomology is considered. The same argument works for flat cohomology.}
Here ${\,}_{\gamma} G$ denotes the twist of $G$ by a cocycle representing $\gamma$, and similarly for $G_{\rm red}$. 
Since the homogeneous space $G/G_{\rm red}$ is purely inseparable over $\Spec(k)$, 
the homogeneous space  ${\,}_{\gamma}G/{\, }_{\gamma} G_{\rm red}$  
is purely inseparable over $\Spec(K)$. (To see this, pass to a splitting field of $\gamma$.)
Thus $({\,}_{\gamma}G/{\, }_{\gamma} G_{\rm red})$  
can have at most one $\Spec(K)$-point. This shows that the fiber of $(i_*)^{-1}(\gamma)$ has at most one element, proving (a). 
If $K$ is perfect, then ${\, }_{\gamma}G(K)$ in $({\,}_{\gamma}G/{\, }_{\gamma} G_{\rm red})$ has exactly one $K$-point. In this case
the fiber of $(i_*)^{-1}(\gamma)$ has exactly one element for every $\gamma \in H^1(K, G_{\rm red})$. This proves (b).
\end{proof}

Recall that an infinitesimal group is a connected $0$-dimensional group. Non-trivial infinitesimal groups exist only in prime characteristic.

\begin{prop} \label{prop.reduced} Let $G$ be an algebraic group over $k$.
Then 

\smallskip
(a) $\rd_k(G) \leqslant \max \{ \rd_k(G_{\rm red}), 1 \}$.

\smallskip
(b) If $G$ be an infinitesimal group over $k$, then $\rd_k(G) \leqslant 1$.

\smallskip
(c) Let $G$ be a $0$-dimensional abelian group over $k$ (not necessarily smooth or connected). Then $\rd_k(G) \leqslant 1$.
\end{prop}

\begin{proof} (a) In view of Proposition~\ref{prop.level-d}(d), it suffices to show that $\rd_k(\alpha) \leqslant \rd_k(G_{\rm red})$
for every field $K/k$ such that $K$ is closed at level $1$ and every $\alpha \in H^1(K, G)$. Indeed, every such field $K$ is perfect;
see Corollary~\ref{cor.level-1}. Thus by Lemma~\ref{lem.non-reduced},
$\alpha$ is the image of some $\beta \in H^1(K, G_{\rm red})$. 
Every field extension of $K$ which splits $\beta$ also splits $\alpha$. This tells us that
\[ \rd_k(\alpha) \leqslant \rd_k(\beta) \leqslant \rd_k(G_{\rm red}), \]
as claimed.

\smallskip
(b) If $G$ is infinitesimal, then $G_{\rm red} = 1$. Thus $\rd_k(G_{\rm red}) = 0$, and $\rd_k(G) \leqslant 1$ by part (a).

\smallskip
(c) Consider the exact sequence 
$1 \to G^0 \to G \to G/G^0 \to 1$.
The group $G^0$ is infinitesimal; thus $\rd_k(G^0) \leqslant 1$ 
by part (b). On the other hand, by~\cite[Exp.~VI$_A$, Proposition 5.5.1]{SGA3I}, $G/G^0$ is \'etale. Since $k$ is algebraically closed, 
this tells us that $G/G^0$ is constant, i.e., is isomorphic to an abstract finite abelian group, viewed as an algebraic group over $k$. 
In particular, 
$\rd_k(G/G^0) \leqslant 1$
by Example~\ref{ex.solvable-group}. Applying Proposition~\ref{prop.group2}(a) to the exact sequence~$1 \to G^0 \to G \to G/G^0 \to 1$,
we obtain $\rd_k(G) \leqslant 1$.
\end{proof}

\begin{prop} \label{prop.abelian} Let $A$ be an abelian variety over $k$. Then $\rd_k(A) \leqslant 1$.
\end{prop}

\begin{proof} Let $K$ be
a field containing $k$. Recall that $H^1(K, A)$ (the Weil-Ch$\hat{\rm a}$telet group of $A_K$) is torsion. 
Thus it suffices to show that $\rd_k \big( H^1(\ast, A)[d] \big)  \leqslant 1$
for every integer $d \geqslant 1$. Examining the exact sequence in cohomology associated to
\[  \xymatrix{  1 \ar@{->}[r] & A[d] \ar@{->}[r] & A \ar@{->}[r]^{\times \, d \quad} & A \ar@{->}[r] & 1 } \]
we conclude that $H^1(\ast, A[d])$ surjects onto 
$H^1(\ast, A)[d]$; see \cite[Section VIII.2]{silverman}~\footnote{$A$ is assumed to be an elliptic curve in~\cite{silverman}, but the same 
argument goes through 
for an abelian variety of arbitrary dimension.}.
Lemma~\ref{lem.functor3}(c) now tells us that
\[ \rd_k(A) = \rd_k \big( H^1(\ast, A)[d] \big)  \leqslant \rd_k \big( H^1(\ast, A[d]) \big) = \rd_k(A[d]). \]
Since $A[d]$ is a $0$-dimensional abelian group over $k$, 
$\rd_k (A[d]) \leqslant 1$ by Proposition~\ref{prop.reduced}(c), and part (c) follows.
\end{proof}

\section{Proof of Theorem~\ref{thm.main1}}
\label{sect.proof-main1}

Setting $\mathcal{F}$ to be the non-abelian cohomology functor $H^1(\ast, G)$
in Proposition~\ref{prop.functor2}(b), we obtain $\rd_k(G) = \rd_{\overline{k}}(G)$ and
$\rd_{k'}(G) = \rd_{\overline{k'}}(G_{\overline{k'}})$, where $\overline{k}$ is the algebraic closure of $k$ and similarly for $k'$.
After replacing $k$ and $k'$ by $\overline{k}$ and $\overline{k'}$, we may assume that $k$ and $k'$ are algebraically closed. 

The following two lemmas will allow us to complete the proof by appealing to Proposition~\ref{prop.functor5}. Lemma~\ref{lem.main1a}
tells us that the conditions of Proposition~\ref{prop.functor5} are satisfied by the non-abelian cohomology functor 
$\mathcal{F} = H^1(\ast, G)$ and thus 
\[ \rd_{k'}(G_{k'}) \leqslant \rd_k(G) \leqslant \{ \rd_{k'}(G_{k'}), \, 1 \} . \]
This yields the desired equality, $\rd_k(G) = \rd_{k'}(G_{k'})$, assuming $\rd_{k'}(G_{k'}) \geqslant 1$. Lemma~\ref{lem.main1b}
shows that this equality also holds when $\rd_{k'}(G_{k'}) = 0$.

\begin{lem} \label{lem.main1a} Let $k$ be algebraically closed field and $G$ be an algebraic group over $k$.
Then the natural map $H^1(E, G) \to H^1(E((t)), G)$ has trivial kernel for every perfect field $E$ containing $k$.
\end{lem}

\begin{proof}
We begin by reducing the problem to the case where $G$ is smooth. Indeed, let $G_{\rm red}$ be the associated smooth group.
Consider the following diagram
\[  \xymatrix{  1 \ar@{->}[r] & H^1(E, G_{\rm red}) \ar@{->}[d] \ar@{->}[r] & \ar@{->}[d] H^1(E, G)   \ar@{->}[r] & 1 \\
 1 \ar@{->}[r] & H^1(E((t)), G_{\rm red})  \ar@{->}[r] & H^1(E((t)), G),  & } \]
where the bottom row is exact by Lemma~\ref{lem.non-reduced}(a) and the top row is exact by Lemma~\ref{lem.non-reduced}(b).
(Recall that we are assuming $E$ to be perfect.) An easy diagram chase shows that if the left vertical map has trivial kernel,
then so does the right vertical map. In other words, if the lemma holds for $G_{\rm red}$, then it also holds for $G$. From now on we will assume that $G$ is smooth.

Now suppose $\alpha \in H^1(E, G)$ lies in the kernel of the map $H^1(E, G) \to H^1(E((t)), G)$. This means that the $G$-torsor
$\pi \colon T \to \Spec(E)$ representing $\alpha$ splits over $\Spec\big( E((t)) \big)$. In other words,  
$\pi_{E[[t]]} \colon T \times \Spec(E[[t]])
\to \Spec(E[[t]])$ has a section $s \colon \Spec(E((t))) \to T \times \Spec(E((t)))$ over the generic point of $\Spec(E[[t]])$.
We would like to show that this section extends to all of $\Spec(E[[t]])$. If we were able to do this, then restricting to the closed point
of $\Spec(E[[t]])$ (i.e., setting $t = 0$), would yield an $E$-point on $T$, showing that $T$ is split, i.e., $\alpha = [T] = 1$
in $H^1(E, G)$, as desired.
To show that $s$ can be extended to all of $\Spec(E[[t]])$, it is natural to appeal to the valuative criterion for properness.
If $G$ is proper over $\Spec(k)$ (i.e., the identity component of $G$ is an abelian variety), then $T$ is proper over $\Spec(K)$,
a desired lifting exists by the valuative criterion for properness, and the proof is complete.
In general, $T$ is not proper over $\Spec(K)$, so the valuative criterion for properness 
does not apply. Nevertheless, we will now show that a variant of this argument still goes through, if we modify $T$ slightly 
as follows.
 
By Chevalley's Structure Theorem~\cite{chevalley, Con} there exists a unique connected smooth normal affine $k$-subgroup 
$N= G^0_{\rm aff}$ of $G^0$ such that the quotient $G^0/N$ is an abelian variety. Since $G$ is smooth and $k$ is 
an algebraically closed field, $N$ is smooth, connected and normal in $G$, and $G/N$ is proper over $\Spec(k)$; 
see~\cite[Theorem 4.2 and Remark 4.3]{brion15}.

Let $B$ be a Borel subgroup (i.e., a maximal connected solvable subgroup) of $N$. Then the homogeneous space 
$N/B$ is proper over $\Spec(k)$. Consequently $G/B$ is proper over $G/N$. Since $G/N$ is proper over $\Spec(k)$, we conclude that
$G/B$ is proper over $\Spec(k)$ and hence, $T/B$ is proper over $\Spec(K)$.
Now the valuative criterion for properness tells us that the section 
\[ s \colon \Spec \big( E((t)) \big) \to T_{\Spec(E[[t]])} \to (T/B)_{\Spec(E[[t]])} \]
extends to $\Spec (E[[t]])$. Restricting to the closed point of $\Spec(E[[t]])$, we obtain an $E$-point on $T/B$.
Denote this point by $p \colon \Spec(E) \to T/B$.
The preimage of this $E$-point under the natural map $T \to T/B$ is a $B$-torsor over $\Spec(E)$.
Since $B$ is connected and solvable, it is special; see Lemma~\ref{lem.group1}(b). Thus this $B$-torsor is split, i.e., has an $E$-point.
This shows that $T$ has an $E$-point, i.e., $T$ is split over $E$, as desired.
\end{proof}

\begin{lem} \label{lem.main1b} Let $k \subset k'$ be algebraically closed fields and $G$ be an algebraic group over $k$.
Then the following conditions are equivalent.

\smallskip
(a) $\rd_{k'}(G) = 0$, 

\smallskip
(b) $G_{k'}$ is a special group,

\smallskip
(c) $G_{k'}$ is affine, and $\ed_{k'}(G_{k'}) = 0$,

\smallskip
(d) $G$ is affine, and $\ed_k(G) = 0$,

\smallskip
(e) $G$ is a special group,

\smallskip
(f) $\rd_{k}(G) = 0$.
\end{lem}

\begin{proof} (a) $\Longleftrightarrow$ (b) and (e) $\Longleftrightarrow$ (f) by Lemma~\ref{lem.group1}(a). 

\smallskip
(b) $\Longrightarrow$ (c): Suppose $G_{k'}$ is special, i.e., $H^1(\ast, G_{k'})$ is the trivial functor. Then clearly 
$\ed_{k'}(G_{k'}) = \ed_{k'} (H^1(\ast, G_{k'})) = 0$. Moreover, a special group is affine; see~\cite[Theorem 4.1]{serre-special}.

\smallskip
(c) $\Longleftrightarrow$ (d): $G$ is affine if and only if $G_{k'}$ is affine. Moreover, since $k$ is algebraically closed
and $G$ is affine group over $k$, we have
$\ed_k(G) = \ed_{k'}(G_{k'})$; see~\cite[Proposition 2.14]{brosnan2007essential} or~\cite[Example 4.10]{tossici2017essential}.

\smallskip
(d) $\Longrightarrow$ (e) by \cite[Proposition 3.16]{merkurjev-survey}.

\smallskip
(f) $\Longrightarrow$ (a): By Proposition~\ref{prop.functor2}(a) with $\mathcal{F} = H^1(\ast, G)$, $\rd_k(G) \geqslant \rd_{k'}(G_{k'})$.
In particular, if $\rd_k(G) = 0$, then $\rd_{k'}(G_{k'}) = 0$.
\end{proof}

\section{Proof of Theorem~\ref{thm.main2}}
\label{sect.proof-main2}

Our proof of Theorem~\ref{thm.main2} will rely on the following proposition.

\begin{prop} \label{prop.dvr} Let $D$ be a discrete valuation ring with fraction field $k$ and residue field $k_0$. 
Let $G$ be a smooth affine group scheme over $D$. Assume that the connected component $G^0$ is reductive 
and the component group $G/G^0$ is finite over $D$. If $\Char(k) = 0$ and $\Char(k_0) = p > 0$, assume further that the absolute 
ramification index of $D$ is $1$. Then
	\[ \rd_{k_0}(G_{k_0}) \leqslant \max \{ \rd_{k}(G_k), \, 1 \}. \]
\end{prop}

Recall that the absolute ramification index of $D$ is defined as $\nu(p)$, where $\nu \colon k^* \to \bbZ$ is the discrete valuation.

\begin{proof} First observe that we may replace $D$ by its completion $\widehat{D}$. Indeed, denote the fraction field of $\widehat{D}$ by
$\widehat{k}$. Then $k \subset \widehat{k}$, and the residue field $k_0$ remains unchanged. By Proposition~\ref{prop.functor2}(a),
$\rd_{\widehat{k}}(G_{\widehat{k}}) \leqslant \rd_k(G_k)$. Thus it suffices to show that  $\rd_{k_0}(G_{k_0}) \leqslant \max \{ \rd_{\widehat{k}}(G_{\widehat{k}}), \, 1 \}$.

After replacing $D$ by $\widehat{D}$, we may assume that $D$ is complete. Under the assumptions of the proposition, $D = W(k_0)$; see~\cite[Sections II.4-5]{serre-lf}.
Here for any field $K_0$ containing $k_0$ we define $W(K_0)$ to be the ring of power series $K_0[[t]]$ if $\Char(k) = \Char(k_0)$ and 
the ring of Witt vectors with coefficients in $K_0$ if $\Char(k) = 0$ and $\Char(k_0) = p > 0$; see~\cite[Section II.6]{serre-lf}.

Now let $K_0$ be a field containing $k_0$ and $\pi \colon T_0 \to \Spec(K_0)$ be a $G_{k_0}$-torsor. 
Set $R = W(K_0)$. Recall that $R$ is a complete local ring relative to a valuation 
$\nu \colon R \to \bbZ$ with residue field $K_0$, extending the valuation on $D$.
By~\cite[XXIV, Proposition 8.1]{SGA3III}, the natural map
$H^1(R, G) \to H^1(K_0, G)$ is bijective. In particular, there exists a $G_R$-torsor $\pi_R \colon T \to \Spec(R)$ 
which restricts to $\pi$ over the closed point $\Spec(K_0) \to \Spec(R)$. Let $K$ be the field of fractions of $R = W(K_0)$
and $\pi_K \colon T_K \to \Spec(K)$ be the restriction of $\pi$ to the generic point of $\Spec(R)$. Our goal is to show that
	\begin{equation} \label{e.dvr} \rd_{l}(T_{K_0}) \leqslant \max \{ \rd_{k}(T_K), \, 1 \}. \end{equation}
This inequality tells us that $\rd_{k_0}(T_{K_0}) \leqslant \max \{ \rd_{k}(G_k), \, 1 \}$. Taking the maximum of 
the left hand side over all $K_0 \in \Fields_{k_0}$ and all $G_{k_0}$-torsors $T_0 \to \Spec(K_0)$, we arrive at $\rd_{k_0}(G_{k_0}) \leqslant \max \{ \rd_{k}(G_k), \, 1 \}$,
as desired.

It remains to prove the inequality~\eqref{e.dvr}.
%
By the definition of $d = \rd_{k}(T_{K})$ there exists a finite field extension $L/K$ such that $L$ splits $T_K$ and
$\lev_{k}(L/K) \leqslant d$. The valuation $\nu$ extends from $K$ and to $L$. Once again, by abuse of notation I will continue 
to denote this extended valuation by $\nu \colon L^* \to \bbZ$. I will also denote the valuation ring for this valuation by $S$ and
the residue field by $L_0$. Now consider the diagram of natural morphisms
\[  \xymatrix{  T_S \ar@{->}[dd] \ar@{->}[rr]  &     &   T_{L} \\ 
                                     & H^1(S, G_S)  \ar@{^{(}->}[r] \ar@{->}[d] &  H^1(L, G_{L})   \\
                T_{L_{0}}         & H^1(L_{0}, G_{L_{0}})   &   } \]
The horizontal map is injective by~\cite[Lemma 3.3(b)]{specialization1}.~\footnote{
\cite[Lemma 3.3(b)]{specialization1} is a variant of the Grothendieck-Serre Conjecture over a Henselian discrete valuation ring. 
A theorem of Nisnevich, establishing the Grothendick-Serre Conjecture in this context, is a key ingredient in our proof of both  Proposition~\ref{prop.dvr} and Theorem~\ref{thm.main2}.}
Note that the assumptions of Theorem~\ref{thm.main2}, that $G^0$ is reductive and $G/G^0$ is finite over $D$ 
(and hence, over $R$ and over $S$), are used to ensure that~\cite[Lemma 3.3(b)]{specialization1} applies.

Since $T_K$ splits over $\Spec(L)$, the injectivity of the horizontal map in the above diagram tells us that that $T$ splits over $\Spec(S)$.
Consequently, $T_{K_0}$ splits over $\Spec(L_0)$. This tells us that
\[ \rd_k(T_{K_{0}}) \leqslant \lev_{k_0}(L_{0}/K_{0}) \leqslant \max \{ \lev_{k} (L/K), \, 1 \} \leqslant
\max \{ d, \, 1 \}, \]
where the middle inequality is given by Proposition~\ref{prop.lev-valuation}. This completes the proof of the inequality~\eqref{e.dvr} and thus
of Proposition~\ref{prop.dvr}.
\end{proof} 

We now proceed with the proof of Theorem~\ref{thm.main2}. If $\Char(k) = \Char(k_0)$, then 
by Theorem~\ref{thm.main1}, $\ed_{k}(G_k) = \ed_{F}(G_F) = \ed_{k_0}(G_{k_0})$, where $F$ is the prime field.
Thus we may assume without loss of generality that $\Char(k) = 0$ and $\Char(k_0) = p > 0$. 
Moreover, we are free to replace $k$ by any field of characteristic $0$ 
and $k_0$ by any field of characteristic $p$. 
%

If $\rd_{k}(G_{k}) \geqslant 1$ for some (and thus every) field of characteristic $0$,
then the desired inequality $\rd_{k_0}(G_{k_0}) \leqslant \rd_k(G_k)$ readily follows from Proposition~\ref{prop.dvr}, applied to
the group scheme $G_D$, where $D = W(k_0)$ = the ring of Witt vectors with coefficients in $k_0$.

The case, where $\rd_{k}(G_{k}) = 0$ needs to be treated separately (as in the previous section).
One again, we are free to choose $k$ to be any field of characteristic $0$ and $k_0$ to be any field of characteristic $p$.
In particular, we may assume that both $k$ and $k_0$ are algebraically closed.
Now by Lemma~\ref{lem.group1}(a), it suffices to show that if $G_{k}$ is special, then $G_{k_0}$ is special. 

Indeed, if $G_{k}$ is special, then $G_{k}$ is connected~\cite[Theorem 4.1]{serre-special}. Hence, $G$ and $G_{k_0}$ are also connected. 
On the other hand, since $k$ and $k_0$ are algebraically closed, we may appeal to the classification of special groups over an algebraically 
closed field due to Grothendieck~\cite[Theorem 3]{grothendieck-special}. According to this classification, 
$G_{k}$ is special if and only if its derived subgroup is a direct product $G_1 \times \ldots \times G_r$, 
where each $G_i$ is a simply connected simple group of type A or C. This property is encoded into the root datum, 
which is the same for $G_{k}$, $G$, and $G_{k_0}$; see~\cite[XXV, Section 1]{SGA3III}. (Note that this is the only step in the proof of Theorem~\ref{thm.main2}
which uses the assumption that $G^0$ is split.) We conclude that $G_k$ is special if and only if $G_{k_0}$ is special. This finishes the proof of 
Theorem~\ref{thm.main2}. 
\qed

\section{Upper bounds on the resolvent degree of a group}
\label{sect.upper-bounds}

Consider an action of a linear algebraic group $G$ on an algebraic variety $X$ (not necessarily connected) defined over a field $k$.
We will say that this action is generically free if there exists a dense $G$-invariant open subset $U \subset X$ such that the scheme-theoretic stabilizer $G_u$ is trivial for every geometric point $u \in U$. In this section we will prove the following.

\begin{prop} \label{prop.hypersurface}
Let $G$ be a closed subgroup of $\PGL_{n}$ defined over $k$. Suppose there exists
a $G$-invariant closed subvariety $X$ of $\bbP^n$ of degree $a$ and dimension $b$.

\smallskip
(a) (cf.~\cite[Proposition 4.11]{wolfson}) If the $G$-action on $X$ is generically free, then \[ \rd_k(G) \leqslant \max \{ b - \dim(G), \,  \rd_k(\Sym_a), \, 1 \}, \] 
where $\Sym_a$ denotes the symmetric group on $a$ letters.

\smallskip
(b) Suppose $\Char(k) \neq 2$ and there exists a $G$-invariant quadric hypersurface $Q \subset \bbP(V)$ of rank $r$
such that $\dim(Q \cap X) = b - 1$.  Assume further the $G$-action on $Q \cap X$ is generically free,
and $b \geqslant \lfloor \dfrac{r+1}{2} \rfloor$.  
Then \[ \rd_k(G) \leqslant \max \{ b - 1 - \dim(G), \, \rd_{k}(\Sym_a), \, 1  \}. \]
\end{prop}

\begin{remark} \label{rem.Sym} (1) Note that $\rd_k(\Sym_a) \geqslant 1$ if $a \geqslant 2$. Thus for $a \geqslant 2$,
the conclusions of parts (a) and (b) simplify to $\rd_k(G) \leqslant \max \{ b - \dim(G), \,  \rd_k(\Sym_a) \}$
and $\rd_k(G) \leqslant \max \{ b - \dim(G) - 1, \,  \rd_k(\Sym_a) \}$, respectively.

\smallskip
(2) By the rank $r$ of $Q$ we mean the rank of some (and thus any) quadratic form defining $Q$. 
The maximal value of $r$ is $n + 1$; it is attained when $Q$ is non-singular. In particular, the condition that 
$b \geqslant \lfloor \dfrac{r+1}{2} \rfloor$
is automatically satisfied if $b \geqslant \lfloor \dfrac{n + 2}{2} \rfloor$. If $X$ is a hypersurface, i.e., $b = n - 1$, then it is automatically satisfied whenever $n \geqslant 3$.

\smallskip
(3) Note that by Theorem~\ref{thm.main1}, $\rd_k(\Sym_a) = \rd_{\bbC} (\Sym_a)$ for any field $k$ of characteristic $0$
and $\rd_k(\Sym_a) \leqslant \rd_{\bbC}(\Sym_a)$ for any field $k$ of positive characteristic. Thus $\rd_k(\Sym_a)$ can be replaced by 
$\rd_{\mathbb C} (\Sym_a)$ in the statement of the proposition.
%
\end{remark}

\begin{example} \label{ex.hypersurface}
(1) If we take $X = \bbP^n$ in part (a), we obtain $\rd_k(G) \leqslant n - \dim(G)$. In particular, if an abstract finite group $G$
has an $n$-dimensional faithful projective representation over $k$, then $\rd_k(G) \leqslant n$. (The $G$-action on $\bbP^n$ is automatically generically free in this case.) In particular, since the alternating group $\Alt_5$ acts faithfully on $\bbP^1$, we obtain $\rd_k(\Alt_5) \leqslant 1$.
Also, since $\Alt_6$ and $\Alt_7$ have complex projective representations of dimension $2$ and $3$, respectively, we deduce classical upper 
bounds $\rd_k(\Alt_6) \leqslant 2$ and $\rd_k(\Alt_7) \leqslant 3$; see \cite[Section 3 and 4]{dixmier}, 
\cite[Theorem 5.6]{farb-wolfson}. Note also that $\rd_k(\Sym_n) = \rd_k(\Alt_n)$ for any $n \geqslant 3$; this follows from Proposition~\ref{prop.group2}(a) applied to the exact sequence $1 \to \Alt_n \to \Sym_n \to \bbZ/ 2 \bbZ \to 1$.

(2) More generally, the classical upper bound $\rd_k(\Sym_n) \leqslant n - 4$ for any $n \geqslant 5$ can be deduced from Proposition~\ref{prop.hypersurface}(b) as follows.
Consider the $(n-1)$-dimensional subspace of $k^n$ given by $x_1 + x_2 + \ldots + x_n = 0$. The group $\Sym_n$ acts on this space 
by permuting the coordinates. This yields an embedding $\Sym_n \hookrightarrow \PGL_{n-2}$. The desired inequality now follows from Proposition~\ref{prop.hypersurface}(b), where we take $X$ to be the cubic hypersurface given by $s_3 = 0$ and $Q$ to be the quadric 
$s_2 = 0$; see Remark~\ref{rem.Sym}(2). Here $s_i$ denotes the $i$th elementary symmetric polynomial.
\end{example} 

None of the upper bounds in Example~\ref{ex.hypersurface} are new; the point here is that they can all be deduced from Proposition~\ref{prop.hypersurface} in a uniform way. The remainder of this section will be devoted to proving Proposition~\ref{prop.hypersurface}. We begin with two lemmas.

\begin{lem} \label{lem.upper1} Let $k$ be a field, $K \in \Fields_k$ be closed at level $d$, and
$\emptyset \neq X \subset \bbP^n$ be a projective variety of degree $\leqslant a$ defined over $K$. 
If $d \geqslant \max \{ \rd_k(\Sym_a), \, 1 \}$, then

\smallskip
(a) $K$-points are dense in $X$.

\smallskip
(b) Assume further that $Q \subset \bbP^n$ be a quadric hypersurface of rank $r$ defined over $K$ and
$\dim(X) \geqslant \lfloor \dfrac{r + 1}{2} \rfloor$. Then $K$-points are dense in $X \cap Q$.
\end{lem}

\begin{proof} 
(a) We argue by induction on $n$. The base case, where $n = 1$, reduces to the assertion that every non-constant homogeneous 
polynomial $f(x, y) \in K[x, y]$ of degree $\leqslant a$ splits into a product of linear factors over $K$.  This assertion follows
from Example~\ref{ex.etale2}. (Recall that $\rd_k(\Sym_n) = \rd_k(\text{\'Et}_n))$; see Example~\ref{ex.etale3}.) 

For the induction step, assume $n \geqslant 2$ and consider the incidence variety 
\[ \xymatrix{  & I = \{ (p, L) \in X \times \widehat{\bbP}^n  \; | \; p \in L \}  \ar@{->}[ld]_{\pi_1} \ar@{->}[rd]^{\pi_2} & \\
 X &  &\widehat{\bbP}^n ,} \]   
where $\widehat{\bbP}^n$ is the dual projective space parametrizing hyperplanes in $\bbP^n$ and $\pi_1$, $\pi_2$ are projections 
to the first and second factor, respectively. Clearly $\pi_1$ is surjective; there is a hyperplane in $\bbP^n$ through every point of $X$.
By the induction assumption, $K$-points are dense $\pi_2^{-1}(L)$ for every $L \in \widehat{\bbP}^{n}(K)$.
Since $K$-points are dense in $\widehat{\bbP}^n$, we conclude that $K$-points are dense in $I$. Projecting them to $X$ via $\pi_1$, 
we see that $K$-points are dense in $X$, as desired.

\smallskip
(b) Since $K$ is closed at level $d \geqslant 1$, every quadratic form splits over $K$. 
That is, $Q$ is the zero locus of a split quadratic form $q(x_0, \ldots, x_n)$ of rank $r$ over $K$. 
Let $m = \lfloor \dfrac{r}{2} \rfloor + (n + 1 - r)$
and $\Gr_q(m, n + 1)$ be the isotropic Grassmannian of maximal isotropic subspaces of $q$. In other words, $\Gr_q(m, n)$ parametrizes 
linear subspaces of (projective) dimension $m - 1$ which are contained in $Q$. Since $q$ is split, $K$-points are dense in $\Gr_q(m, n)$.
Consider the incidence variety
\[ \xymatrix{  & I = \{ (p, L) \in (Q \cap X) \times \Gr_q(m, n)  \; | \; p \in L \}  \ar@{->}[ld]_{\pi_1} \ar@{->}[rd]^{\pi_2} & \\
 Q \cap X &  & \Gr_q(m, n).} \]  
Note that there exists a maximal isotropic subspace through every point of $Q$; in particular, $\pi_1$ is surjective. On the other hand,
our assumption that $\dim(X) \geqslant  \lfloor \dfrac{r+1}{2} \rfloor$ ensures that $\dim(L) + \dim(X) = (m - 1) + \dim(X) \geqslant n$ and thus
$L \cap X \neq \emptyset$ every $L \in \Gr_q(m, n)$. This tells us that
$\pi_2$ is surjective. The fiber $\pi_1^{-1}(L) = L \cap X$ of a $K$-rational point $L$ of $\Gr_q(m, n)$ is a closed subvariety
of degree $d$ in $L \simeq \bbP^{m-1}$ defined over $K$. By part (a), $K$-points are dense in every such fiber. Since $K$-points are dense in
$\Gr_q(m, n)$, we conclude that $K$-points are dense in $I$. Projecting them to $Q \cap X$ via $\pi_1$, 
we see that $K$-points are dense in $Q \cap X$ as well.
\end{proof}

\begin{lem} \label{lem.upper3} 
Consider a generically free action of an algebraic group $G$ on an algebraic variety $X$ defined over a field $k$. Supposed
that $K$-points are dense in the twisted variety ${\, }_{T} \, X$ for every field $K$ containing $k$, closed at level $d$
and every $G$-torsor $T \to \Spec(K)$. Then $\rd_k(G) \leqslant \max \{ d, \, \dim(X) - \dim(G) \}$.
\end{lem}

\begin{proof} Let $K/k$ be a field closed at level $d$ and $T \to \Spec(K)$ be a $G$-torsor.
By Proposition~\ref{prop.level-d}(c) it suffices to show that 
\begin{equation} \label{e.twisted} \rd_k(T) \leqslant \dim(X) - \dim(G) .
\end{equation}
After replacing $X$ by a suitable $G$-invariant union of its irreducible components, we may assume that
$G$ transitively permutes the irreducible components of $X$. In this case there exists
a $G$-invariant dense open subvariety $U \subset X$, which is the total space of a $G$-torsor $U \to B$.
By our assumption ${\, }_{T} \, U$ has a $K$-point. Equivalently, there exists a $G$-equivariant map $T \to U$
defined over $k$; see, e.g., \cite[Proof of Theorem 1.1(a) on p. 508]{duncan2015versality}. 
This implies that $\ed_k(T) \leqslant \dim(B) = \dim(X) - \dim(G)$. 
The desired inequality~\eqref{e.twisted} now follows from the inequality $\rd_k(T) \leqslant \ed_k(T)$ 
of Remark~\ref{rem.group2.5}.
\end{proof}

\begin{proof}[Proof of Proposition~\ref{prop.hypersurface}] (a) Set $d = \max \{ \rd_k(\Sym_a), \, 1 \}$.
Suppose a field $K \in \Fields_k$ is closed at level $d$ and $T \to \Spec(K)$ is a $G$-torsor. Since $d \geqslant 1$, 
$K$ is solvably closed; see Corollary~\ref{cor.level-1}. By Lemma~\ref{lem.upper3} it suffices to show 
that $K$-points are dense in ${\, }_T X$.

Note that the $G$-equivariant closed immersion $X \hookrightarrow \bbP(V)$ induces a natural closed immersion
${\, }_TX \hookrightarrow {\, }_T \bbP(V)$ of $K$-varieties. Here ${\, }_T \bbP(V)$ is a Brauer-Severi variety over $K$. 
Since $K$ is solvably closed, every Brauer-Severi variety over $K$ is split. Thus ${\, }_T X$ is a closed 
subvariety of  $\bbP(V)_K$. The degree of ${\, }_T X$ in $\bbP(V)_K$ is $a$, same as the degree of $X$ 
in $\bbP(V)$. (To see this, pass to $\overline{K}$.) By Lemma~\ref{lem.upper1}(a), $K$-points are dense in ${\, }_T X$.
The desired inequality, $\rd(G) \leqslant \{ d, \, \dim(X) - \dim(G) \}$, now follows from Lemma~\ref{lem.upper3}.

\smallskip
(b) The argument here is the same as in part (a), with Lemma~\ref{lem.upper1}(b) used to show that $K$-points are dense in
${\, }_T (Q \cap X)$.
\end{proof}

\section{Upper bounds on the resolvent degree of some reflection groups}
\label{sect.reflection}

The purpose of this section is to prove the following.

\begin{prop} \label{prop.weyl} Let $W$ be Weyl group of the simple Lie algebra (or equivalently, a simple algebraic group) 
of type $E_i$. Here $i = 6$, $7$ or $8$. Let $k$ be an arbitrary field. Then $\rd_k(W) \leqslant i - 3$.
\end{prop}

The inequality $\rd_k(W) \leqslant 5$, where $W$ is the Weyl group of $E_8$, will play an important role in the proof of Theorem~\ref{thm.main3}.
We will only supply a proof of Proposition~\ref{prop.weyl} in this case (for $i = 8$).
The other two inequalities (where $i = 6$ and $7$) will not be used in this paper. They are proved by a minor modification of the same argument;
we leave the details as an exercise for the reader. 

Note that by Theorem~\ref{thm.main1}, $\rd_{\bbC}(W_i) = \rd_{\bbQ}(W_i) = \rd_k(W_i)$ for any field $k$ of characteristic $0$.
Moreover, by Theorem~\ref{thm.main2}, $\rd_k(W_i) \leqslant \rd_{\bbC}(W_i)$ for any field $k$ of characteristic $p$.
Thus the purpose of proving Theorem~\ref{prop.weyl}, we may assume that $k = \bbC$. This places us into the setting of Springer's 
classic paper on complex reflection groups~\cite{springer-reflection}.

We now proceed with the proof of Proposition~\ref{prop.weyl} for $i = 8$ and $k = \mathbb C$.
Consider the natural representation $W \hookrightarrow \GL(V) = \GL_8$ where $V$ is a Cartan subalgebra of $E_8$.
The kernel $Z$ of this representation is the center of $W$; it is a cyclic group of order $2$. 
We will denote the non-trivial element of $Z$ by $z$ and the image of $W$ in $\PGL_8$ by $\overline{W} = W/Z$. 

Recall that the ring of invariants $\bbC[V]^W$ is a polynomial 
ring over $\bbC$ 
in $8$ variables. The generators $f_2$, $f_8$, $f_{12}$, $f_{14}$, $f_{18}$, $f_{20}$, $f_{24}$ and $f_{30}$ are called basic invariants;
each $f_i$ is a homogeneous $G$-invariant polynomial of degree $i$. These basic invariants are not unique but their degrees
are. That is, if $\bbC[V]^G$ is generated by $8$ homogeneous elements $g_1, \ldots, g_8$,
then the degrees of $g_1, \ldots, g_8$ are 
\begin{equation} \label{e.fund-degrees} 2, \, 8, \, 12, \, 14, \, 18, \, 20, \, 24, \,  30. \end{equation}
These integers are called the fundamental degrees of $W$.

Our strategy is to apply Proposition~\ref{prop.hypersurface}(b) with
$G = \overline{W}$, $X \subset \bbP(V) = \bbP^7$ the hypersurface $f_8 = 0$ and $Q \subset \bbP^7$ the quadric hypersurface $f_2 = 0$. Denote the
affine cones of $Q$ and $X$ by $Q^{\rm aff}$ and $X^{\rm aff}$, respectively.

\begin{lem} \label{lem.e8-1}
(a) $W$ transitively permutes the irreducible components of $Q \cap X$ (or equivalently, the irreducible components of
$Q^{\rm aff} \cap X^{\rm aff}$).

\smallskip
(b) Each irreducible component of $Q \cap X$ is of dimension $5$.
\end{lem} 

\begin{lem} \label{lem.e8-2}
The action of $\overline{W}$ on $Q \cap X$ is generically free.
\end{lem}

\smallskip \noindent
Assume for a moment that we have established these two lemmas. Then Proposition~\ref{prop.hypersurface}(b) tells us that
\[ \rd_{\bbC}(\overline{W}) \leqslant \max \{ \dim(X) - 1, \, \rd_{\bbC}(\Sym_8) \} = \max \{ 5, \, 4 \} = 5. \]
Here I used the fact that $\rd_{\bbC}(\Sym_8) \leqslant 4$; see \cite[Theorem 5.6]{farb-wolfson} or Example~\ref{ex.hypersurface}(b).
Applying Proposition~\ref{prop.group2}(a) to the exact sequence $1 \to Z \to W \to \overline{W} \to 1$, we conclude that
\[ \rd_{\bbC}(W) \leqslant \max \{ \rd_{\bbC}(\overline{W}), \, \rd_{\bbC}(Z) \} \leqslant \max \{ 5, \, 1 \} = 5, \]
as desired. It thus remains to prove Lemmas~\ref{lem.e8-1} and~\ref{lem.e8-2}.

\begin{proof}[Proof of Lemma~\ref{lem.e8-1}] 
The natural inclusion $\bbC[f_2, f_8, \ldots, f_{30}] = \bbC[V]^W \hookrightarrow \bbC[V]$
induces the categorical quotient map $\pi \colon V \to \bbA^8$ given by $\pi \colon v \to (f_2(v), f_8(v), \ldots, f_{30}(v))$.
Note that $\pi$ is a finite morphism, and the fibers of $\mathbb C$-points of $\mathbb A^8$ are precisely the $W$-orbits in $V$. 
By definition, $Q^{\rm aff}$ and $X^{\rm aff}$ the preimages of coordinate hyperplanes $H_1$ and $H_2$ in $\mathbb A^8$ given by
$x_1 = 0$ and $x_2 = 0$, respectively. Both (a) and (b) now follows from the fact that $H_1 \cap H_2 \simeq \mathbb A^6$ is irreducible
of dimension $6$. 
\end{proof}

\begin{proof}[Proof of Lemma~\ref{lem.e8-2}] 
Assume the contrary: the $\overline{W}$-action on $Q \cap X$ is not generically free. This means that $Q^{\rm aff} \cap X^{\rm aff}$ 
is covered by the union of eigenspaces $V(g, \zeta)$, where $g$ ranges over $W \setminus Z$ and $\zeta$ ranges over 
the roots of unity in $\mathbb C$. Here
\[ V(g, \zeta) = \{ v \in V \, | \, g(v) = \zeta v \} \]
stands for the $\zeta$-eigenspace of $g$, as in~\cite{springer-reflection}.

If $\zeta$ is a primitive root of unity of degree $d$, then $\dim(V(g, \zeta)) \leqslant a(d)$, where $a(d)$ is the number of fundamental degrees~\eqref{e.fund-degrees} divisible by $d$; see \cite[Theorem 3.4]{springer-reflection}.
By inspection we see that $a(d) \leqslant 4$ for any $d \geqslant 3$, with equality for $d = 3, 4, 6$. 
Thus the union of eigenspaces 
\[ \bigcup^{g \in W \setminus Z}_{\operatorname{deg}(\zeta) \geqslant 3} \, V(g, \zeta) \]
is at most $4$-dimensional.
Since every irreducible component of $Q^{\rm aff} \cap X^{\rm aff}$ is of dimension $6$
(see Lemma~\ref{lem.e8-1}(b)), $Q^{\rm aff} \cap X^{\rm aff}$ is, in fact, covered by 
the union of $V(g, \pm 1) = V^g$, as $g$ ranges over $W \setminus Z$. 
Since $V(g, -1) = V(zg, 1)$ for each $g$, we conclude that  
\[ V^{\rm non-free} = \cup_{g \in W \setminus Z} \, V(g, 1) \]
covers one of the irreducible components of $Q^{\rm aff} \cap X^{\rm aff}$.
Clearly $V^{\rm non-free}$ is $W$-invariant.
By Lemma~\ref{lem.e8-1}(a), if it covers one irreducible component of $Q^{\rm aff} \cap X^{\rm aff}$, it covers all of them. 
In other words, \[ Q^{\rm aff} \cap X^{\rm aff} \subset \bigcup_{1 \neq g \in W}  V(g, 1) . \] 
Thus in order to produce a contradiction, it suffices to exhibit one point $v \in V$ such that 

\smallskip
(i) $\Stab_W(v) = \{ 1 \}$ or equivalently, $v \not \in V(g, 1)$ for any $1 \neq g \in W$, and  

\smallskip
(ii) $v \in Q^{\rm aff} \cap X^{\rm aff}$ of equivalently, $f_2(v) = f_8(v) = 0$.

\smallskip
By \cite[p. 177, Table 3]{springer-reflection}, $W$ has a regular element of order $3$. This means that $V(g, \zeta_3)$ contains a regular vector
$v$, where $\zeta_3$ is a primitive cube root of unity. Recall that a vector in $V$ is called regular if it is not contained in any reflecting hyperplane, and that for any regular vector $v$, the stabilizer $\Stab_W(v) = \{ 1 \}$; see \cite[Proposition 4.1]{springer-reflection}.
Moreover, if $f_d$ is one of the fundamental invariants, then
\[ f_d(v) = f_d(gv) = f_d(\zeta_3 v) = \zeta_3^d f_d(v). \]
In particular, $f_d(v) = 0$ when $d = 2$ and $8$. Thus the regular vector $v$ satisfies conditions (i) and (ii).
This completes the proof of Lemma~\ref{lem.e8-2} and thus of Proposition~\ref{prop.weyl} for $i = 8$.
\end{proof}

\section{Proof of Theorem~\ref{thm.main3}} 
\label{sect.proof-main3}

Once again, by Theorem~\ref{thm.main1}, we may replace $k$ by its algebraic closure and 
thus assume without loss of generality that $k$ is algebraically closed.

\smallskip
{\bf Reduction to the case, where $G$ is smooth.} Recall that $\rd_k(G) \leqslant \max \{ \rd_k(G_{\rm red}), 1 \}$
by Proposition~\ref{prop.reduced}(a). Thus in order to prove Theorem~\ref{thm.main3}
for $G$, it suffices to prove it for $G_{\rm red}$.

\smallskip
{\bf Reduction to the case, where $G$ is affine.} We may now assume that $G$ is smooth.
By Chevalley's structure theorem~\cite{chevalley, Con} there exists 
a unique connected smooth normal affine $k$-subgroup $G_{\rm aff}$ of $G$ such that the quotient $G/G^{\rm aff}$ is an abelian variety.
By Proposition~\ref{prop.abelian}, $\rd_k(G/G^{\rm aff}) \leqslant 1$.
Applying  Proposition~\ref{prop.group2}(a) to the exact sequence $1 \to G_{\rm aff} \to G \to G/G^{\rm aff} \to 1$, we obtain
$\rd_k(G) \leqslant \max \{ \rd_k(G_{\rm aff}), \,  1 \}$. Thus in order to prove Theorem~\ref{thm.main3} for $G$, 
it suffices to prove it for 
$G_{\rm aff}$.

\smallskip
{\bf Reduction to the case, where $G$ is semisimple.} We may now assume that $G$ is affine. 
Let $\Rad(G)$ be the radical of $G$, i.e.,
the largest connected solvable normal subgroup of $G$. Denote the quotient (semisimple) group by $G^{ss}$ and 
consider the natural exact sequence
$1 \to \Rad(G) \to G \to G^{ss} \to 1$.
By Lemma~\ref{lem.group1}(b), $\rd_k ( \Rad(G) ) = 0$. Proposition~\ref{prop.group2}(a)
now tells us that $\rd_k(G) \leqslant \rd_k(G^{ss})$. Thus in order to prove Theorem~\ref{thm.main3} for $G$, 
it suffices to prove it for $G^{ss}$.

\smallskip
{\bf Reduction to the case, where $G$ is almost simple.}  
We will now assume that $G$ is semisimple. Then $G$ isogenous to the direct product $\widetilde{G} = G_1 \times \ldots \times G_r$ of its minimal 
connected normal subgroups. That is, there exists a central exact sequence
\[ 1 \to A \to \widetilde{G} \to G \to 1 , \]
where $A$ is a finite subgroup of a maximal torus of $G_1$; see \cite[Section 9.6.1]{springer}. Since we are assuming that
$k$ is algebraically closed, this tells us that $A$ is a finite diagonalizable group. Hence, $\rd_k(A) \leqslant 1$ by Proposition~\ref{prop.group2}(c).
The minimal connected normal subgroups 
$G_1, \ldots, G_r$ are (almost) simple; see \cite[Section 27.5]{humphreys}.
Proposition~\ref{prop.group2}(d) now tells us that it suffices to prove Theorem~\ref{thm.main3} 
for $\widetilde{G} = G_1 \times \ldots \times G_r$. Applying 
Proposition~\ref{prop.group2}(b) recursively, we see that
\[ \rd_k(G_1 \times \ldots \times G_r) = \max \, \{ \, \rd_k(G_1), \ldots, \rd_k(G_r) \, \}. \]
Thus in order to prove Theorem~\ref{thm.main3} for $G$, it suffices to prove it for each (almost) simple 
group $G_i$.

 \smallskip
 From now on we will assume that $G$ is (almost) simple. To complete the proof of Theorem~\ref{thm.main3}, it remains to establish the following.

\begin{prop} \label{prop.simple} Let $k$ be an algebraically closed field and $G$ an almost simple group defined over $k$.
Then (a) $\rd_k(G) \leqslant 5$ if $G$ is of type $E_8$ and (b) $\rd_k(G) \leqslant 1$ if $G$ is of any other type.
\end{prop}

\begin{proof} 
(a) Let $G$ be a simple group of type $E_8$ and let $W_8$ be the Weyl group of $G$. Then 
\[ \rd_k(G) \leqslant \rd_k(W_8) \leqslant 5 , \]
where the first inequality is given by Corollary~\ref{cor.group3}, and the second by Proposition~\ref{prop.weyl}.

\smallskip
(b) Tits~\cite[Section 2]{tits-resume} showed that if $G$ is a simple group of any type other than $E_8$, then $G$ has
no non-trivial torsors over any field $K$, closed under taking radicals. (Note that \cite{tits-resume} is reprinted in~\cite{tits-collected}.)
In particular, there are no non-trivial $G$-torsors over any field $K$ closed at level $1$. 
By Proposition~\ref{prop.level-d}(b) this implies that $\rd_k(G) \leqslant 1$, as claimed.

For the sake of completeness we will give a short direct proof of part (b), using the terminology of this paper. 
We begin with two preliminary observations. First, recall that every almost simple algebraic group is defined over $\mathbb{Z}$. 
Using Theorems~\ref{thm.main1} and~\ref{thm.main2}, we may assume without loss of generality 
that $k = \bbC$ is the field of complex numbers. This assumption will allow us to avoid 
some of the subtle points of the arguments in~\cite[Section 2]{tits-resume} which only come up in prime characteristic. 

The second observation is that if $G_1$ and $G_2$ are almost simple groups of the same type, then they are isogenous and hence,
by Proposition~\ref{prop.group2}(d), $\rd_{\mathbb C}(G_1) \leqslant  \max \, \{  \rd_{\mathbb C}(G_2), \, 1  \}$ and
$\rd_{\mathbb C}(G_2) \leqslant  \max \, \{  \rd_{\mathbb C}(G_1), \, 1  \}$. Consequently,
Proposition~\ref{prop.simple}(b) holds for $G_1$ if and only if it holds for $G_2$. In other words, it suffices 
to prove that $\rd_{\bbC}(G) \leqslant 1$ for one almost simple group $G$ of each type (other than $E_8$).

\smallskip
{\bf $G$ is of type $A_r$ or $C_r$.} Here we can take $G$ to be $G = \SL_{r+1}$ and $G = \Sp_{2r}$,
respectively. By Lemma~\ref{lem.group1}(c), in both cases, $\rd_k(G) = 0$. 

\smallskip
{\bf If $G$ is of type $B_r$ or $D_r$.} Here we take $G$ to the the special orthogonal group 
$G = \SO_n$, which is of type $B_r$ if $n = 2r + 1$ and of type $D_{r}$ if $n = 2r$.
By~\cite[ (29.29)]{knus1998involutions}, $H^1(K, G)$ can be represented by $n$-dimensional 
quadratic forms $q$ of discriminant $1$ over $K$. In a suitable basis, $q(x_1, \ldots, x_n) = a_1 x^2 + \ldots + a_n x_n^2$
for some $a_1, \ldots, a_n$. Thus $q$ splits over $L = K(\sqrt{a_1}, \ldots, \sqrt{a_n})$. Clearly $\rd_{\bbC}(L/K) \leqslant 1$,
and thus $\rd_{\bbC}(G) \leqslant 1$, as claimed.

\smallskip
{\bf $G$ is of type $G_2$ and $F_4$.}  In both cases
the only primes dividing $|W|$ are $2$ and $3$. By Burnside's 
Theorem, $W$ is solvable~\footnote{One can also see check this directly, without appealing to Burnside's theorem.}.
Thus $\rd_k(G) \leqslant \rd_k(W) \leqslant 1$, where the first inequality follows from Corollary~\ref{cor.group3}
and the second from Example~\ref{ex.solvable-group}.

\smallskip
{\bf $G$ is a simply connected group of type $E_6$.} 
By \cite[Example 9.12]{garibaldi}, $G$ has a subgroup $S$ isomorphic to
$F_4 \times \mu_3$ such that the map $H^1(K, S) \to H^1(K, G)$ 
is surjective; see~\cite[Section 23]{garibaldi}. Here $F_4$ denotes the simply connected group of type $F_4$.
By Lemma~\ref{lem.functor3}(c), $\rd_{\bbC}(S) = \rd_{\bbC} \, H^1(\ast, S) \geqslant \rd_{\bbC} \, H^1(\ast, G) = \rd_{\bbC}(G)$.
Since we know that $\rd_{\bbC}(F_4) \leqslant 1$, 
\[ \rd_{\bbC}(G) \leqslant \rd_{\bbC}(S) = \rd_{\bbC}(F_4 \times \mu_3) = \max \{ \rd_{\bbC}(F_4), \rd_{\bbC}(\mu_3) \} = 1. \]

\smallskip
{\bf $G$ is a simply connected group of type $E_7$.} 
By \cite[Example 12.3]{garibaldi}, $G$ has a subgroup $\widetilde{S}$ isomorphic to
$E_6 \rtimes \mu_4$ such that the map $H^1(K, \widetilde{S}) \to H^1(K, G)$ 
is surjective; see~\cite[Section 23]{garibaldi}. Here $E_6$ denotes the simply connected group of type $E_6$.
Once again, by Lemma~\ref{lem.functor3}(c), $\rd_{\bbC}(\widetilde{S}) \geqslant \rd_{\bbC}(G)$.
Since we know that $\rd_{\bbC}(E_6) \leqslant 1$, we conclude that
$\rd_{\bbC}(G) \leqslant \rd_{\bbC}(\widetilde{S}) = \rd_{\bbC}(E_6 \rtimes \mu_4) = 
\max \{ \rd_{\bbC}(E_6), \rd_{\bbC}(\mu_4) \} = 1$.

\smallskip
This completes the proof of Proposition~\ref{prop.simple} and thus of Theorem~\ref{thm.main3}.
\end{proof}

\begin{remark}  For simply connected groups $G$ 
of type $G_2$, $F_4$, $E_6$ and $E_7$, the inequality $\rd_{\bbC}(G) \leqslant 1$ of Proposition~\ref{prop.simple}(b) can also be deduced from
a theorem of Garibaldi which asserts that
for these groups the Rost invariant $H^1(\ast, G) \to H^3(\ast, \bbZ/n_G\bbZ(2))$ has trivial kernel; 
see~\cite[Theorem 0.5]{garibaldi} or~\cite{chernousov-rost}.
\end{remark}

\section{Can the inequality of Theorem~\ref{thm.main3} be strengthened?}
\label{sect.serre}

Recall that Conjecture~\ref{conj.main5} asserts that the inequality $\rd_k(G) \leqslant 5$ of Theorem~\ref{thm.main3} can be 
strengthened to $\rd_k(G) \leqslant 1$. In this final section we will show that this conjecture
follows from a positive answer to a long-standing open question of Serre~\cite[Question 2]{serre-pp} stated below.

\begin{question} \label{conj.serre} Let $K$ be a field, $H$ be a smooth algebraic group over $K$, 
and $T \to \Spec(K)$ be a $H$-torsor. If $K_i/K$ are finite extensions of $K$ of relatively prime degrees, i.e., 
$\operatorname{gcd}([K_i: K]) = 1$, and each $K_i$ splits $T$, then $T$ is split over $K$.
\end{question}

For a detailed discussion of Question~\ref{conj.serre}, we refer the reader to~\cite{totaro-E8}.

\begin{prop} \label{prop.serre}
Assume that Question~\ref{conj.serre} has a positive answer in the following special situation: 
$K$ is a solvably closed field containing $\bbC$ and $H = (E_8)_K$ is the split simple group of type $E_8$ over $K$.
Then $\rd_k(G) \leqslant 1$ for every field $k$ and every connected algebraic group $G$ over $k$.
\end{prop}

\begin{proof} 
It suffices to show that, under the assumption of the proposition, the inequality $\rd_k(E_8) \leqslant 5$ of Proposition~\ref{prop.simple}(a)
can be strengthened to $\rd_k(E_8) \leqslant 1$. If we can do this, then the argument of Section~\ref{sect.proof-main3} will go through unchanged to show that
$\rd(G)\leqslant 1$ for every field $k$ and every connected algebraic group $G$ over $k$.

By Theorems~\ref{thm.main1} and~\ref{thm.main2} we may further assume that $k = \mathbb C$, as we did in the proof of Proposition~\ref{prop.simple}
in the previous section.
By Proposition~\ref{prop.level-d}(b) it suffices to show that 
every $E_8$-torsor $T \to \Spec(K)$ is split for every field $K \in \Fields_C$, closed at level $1$ (over $\bbC$). 
In fact, we will show that this is the case whenever $K$ is solvably closed; cf.~Corollary~\ref{cor.level-1}. 
By a theorem of Tits~\cite{tits-splitting}, $T$ is split by a finite field extension
$K_{\geqslant 7}/K$ such that 
\begin{equation} \label{e.2-3-5}
\text{the only primes dividing $[K_{\geqslant 7}: K]$ are $2$, $3$ and $5$;}
\end{equation}
see also \cite{totaro-E8}.~\footnote{Note that this step is valid for every $K$; we do not use the assumption that $K$ is solvably 
closed here.}

Now observe that since $K$ is solvably closed, the Norm Residue Isomorphism Theorem
tells us that $H^d(K, \mu_n) = 1$ for every $d, n \geqslant 1$; cf.~Remark~\ref{rem.hn}.
In particular, the class of $T$ lies in the kernel of the Rost Invariant 
$R \colon H^1(K, E_8) \to H^3(K, \mu_{60})$. Theorems of Chernousov now tell us that 
\[ \text{$T$ is split by a finite extension $K_3/K$ such that $3 \, \not | \; [K_3: K]$;} \]
and
\[ \text{$T$ is split by a finite extension $K_5/K$ such that $5 \, \not | \; [K_5: K]$;} \]
see~\cite{chernousov-rost3, chernousov-rost5}. Finally, $T$ also lies in the kernel of the Semenov invariant
\[ H^1(\ast, E_8)_0 \to H^3(\ast, \mu_2), \]
where
$H^1(\ast, E_8)_0$ denote the kernel of the mod 4 Rost invariant, $15R$. Consequently, by~\cite[Theorem 8.7]{semenov}
\[ \text{$T$ is split by a finite extension $K_2/K$ such that $2 \, \not | \; [K_2: K]$.} \]
In summary, $T$ can be split by finite extensions $K_2$, $K_3$, $K_5$ 
and $K_{\geqslant 7}$ of $K$ whose degrees are relatively prime.
The assumption of the proposition now tells us that
$T$ is split over $K$, as desired.
\end{proof} 

\begin{remark} Note that the Semenov invariant is only defined in characteristic $0$.
In prime characteristic our proof of Proposition~\ref{prop.serre} relies on
Theorem~\ref{thm.main2}.
\end{remark}

\section*{Acknowledgement} The author is grateful to Jesse Wolfson for helpful detailed comments on an earlier version of this paper.


\end{document}